\documentclass[12pt]{extarticle}
\usepackage{amsmath, amsthm, amssymb, mathtools, hyperref,color}
\usepackage{graphicx}
\usepackage{caption}
\usepackage{subcaption}
\usepackage{algorithm}
\usepackage{algpseudocode}
\usepackage{lipsum}

\makeatletter
\newenvironment{breakablealgorithm}
  {
   \begin{center}
     \refstepcounter{algorithm}
     \hrule height.8pt depth0pt \kern2pt
     \renewcommand{\caption}[2][\relax]{
       {\raggedright\textbf{\fname@algorithm~\thealgorithm} ##2\par}%
       \ifx\relax##1\relax 
         \addcontentsline{loa}{algorithm}{\protect\numberline{\thealgorithm}##2}%
       \else 
         \addcontentsline{loa}{algorithm}{\protect\numberline{\thealgorithm}##1}%
       \fi
       \kern2pt\hrule\kern2pt
     }
  }{
     \kern2pt\hrule\relax
   \end{center}
  }
\makeatother

\newtheorem{theorem}{Theorem}
\numberwithin{theorem}{section}
\newtheorem{proposition}[theorem]{Proposition}
\newtheorem{lemma}[theorem]{Lemma}
\newtheorem{corollary}[theorem]{Corollary}

\newtheorem{remark}[theorem]{Remark}

\newtheorem{example}[theorem]{Example}

\algnewcommand\algorithmicforeach{\textbf{for each}}
\algdef{S}[FOR]{ForEach}[1]{\algorithmicforeach\ #1\ \algorithmicdo}

\begin{document}
\title{Poincar\'e Series of Divisors on Graphs and Chains of Loops: Rationality and Algorithms}
\author{Madhusudan Manjunath\footnote{the author was supported by a MATRICS grant of the Department of Science and Technology (DST), India during the course of this work}}
\maketitle

We study Poincar\'e series associated to a finite collection of divisors on i. a finite graph and  ii. a certain family of metric graphs called chain of loops.  Our main results are proofs of rationality of the Poincar\'e series and algorithms for computing it in both these cases. The main tools used in the proof of rationality are the following. For graphs, we study a certain homomorphism from a free Abelian group of finite rank to the direct sum of the Jacobian of the graph and the integers. For chains of loops, our main tool is an analogue of Lang's conjecture for Brill-Noether loci on a chain of loops and adapts the proof of rationality of the Poincar\'e series of divisors on an algebraic curve (over an algebraically closed field of characteristic zero). Our algorithms are based on a closer study of the objects involved in the proof of rationality, for instance, computing the fibres of certain homomorphisms and lattice point enumeration in rational polyhedra.  

\section{Introduction}

Let $L$ be a line bundle on an algebraic variety $X$, a fundamental problem in algebraic geometry called the \emph{Riemann-Roch problem} is to compute the dimension of the space of global sections $h^0(L^ n)$ of the powers of $L$ for large $n$.  A closely related problem is that of  the rationality of the generating function $\sum_{n=1}^{\infty}h^0(L^ n)z^n$.
This generating function is called the  \emph{Poincar\'e series of $L$}.  We refer to the work of Cutkosky and Srinivas \cite{CutSri93} for more details on this topic.
 Cutkosky \cite{Cut03} studied the following multigraded generalisation of the Poincar\'e series  by fixing a finite collection of line bundles $L_1,\dots,L_k$ on $X$ and considering the generating function $\sum_{(n_1,\dots,n_k) \in \mathbb{N}^k}h^0(L_1^{n_1} \otimes L_2^{n_2} \otimes \cdots \otimes L_k^{n_k}) z_1^{n_1} \cdots z_k^{n_k}$ called the \emph{Poincar\'e series of $L_1,\dots,L_k$\footnote{We shall take $\mathbb{N}$ to be the set of non-negative integers throughout the paper.}}. A subtle aspect of this theory is that the Poincar\'e series turns out to be rational (for all choices of $L_1,\dots,L_k$) on smooth curves over an algebraically closed field of characteristic zero but not necessarily rational for smooth curves over an algebraically closed field of positive characteristic and for singular curves. 

In the following, we consider analogues of  Poincar\'e series of divisors on finite graphs and their metrized version, namely compact metric graphs (also known as abstract tropical curves).

\subsection{Poincar\'e Series of Divisors on a Graph}

Given a finite sequence of divisors $D_1,\dots,D_k$ on a finite, connected, loop-free, multigraph $G$. Consider the formal sum: 
 
 \begin{center}
$P_{G}(z_1,\dots,z_k)=\sum_{(n_1,\dots,n_k) \in \mathbb{N}^k}(r_G(n_1D_1+\dots+n_kD_k)+1)z_1^{n_1} \cdots z_k^{n_k}$
\end{center}

where $r_G(D)$ is the rank of the divisor $D$ on the graph $G$.  We refer to this as the \emph{Poincar\'e series associated to divisors $D_1,\dots,D_k$ on $G$}. Note that the Poincar\'e series depends not only on the graph but also on the divisors $D_1,\dots,D_k$.
For easy readability, we suppress this dependence  on the divisors while denoting Poincar\'e series and other related objects. 
  
Using the inequality that the rank of a divisor is at most its degree for every divisor of non-negative degree, it follows that $P_{G}$ is absolutely convergent in the region $\{ (z_1,\dots,z_k)|~|z_i|<1 \textsl{ for all } i \}$.  A natural question in this context is whether $P_{G}(z_1,\dots,z_k)$ is a rational function\footnote{ Recall that a power series in $z_1,\dots,z_k$ is called \emph{rational} if there exists a rational function $f/g$ where  $f,~g \in \mathbb{C}[z_1,\dots,z_k]$ such that the power series agrees with this rational function at every $(z_1,\dots,z_k) \in \mathbb{C}^k$ where it is absolutely convergent.}.   We answer this question in the affirmative, more precisely we show the following.

\begin{theorem}\label{grapoinrat_theo} {\rm ({\bf Rationality of Poincar\'e Series of Divisors on Graphs})}
For any finite, connected, multigraph $G$ without loops and any finite sequence of divisors $D_1,\dots,D_k$ on $G$, the Poincar\'e series $P_{G}(z_1,\dots,z_k)$ is rational.  
\end{theorem}

{\bf Main Ingredients of the Proof:}  A key ingredient is the rationality of lattice point enumerating functions in rational polyhedra (see \cite{Bar99} for a detailed treatment).  Other key ingredients are i. the Jacobian group ${\rm Jac}(G)$ of $G$, in particular its finiteness, ii. the group homomorphism $\phi_{G}:\mathbb{Z}^{k} \rightarrow {\rm Div}(G)/{\rm Prin}(G)$ given by $(n_1,\dots,n_k) \rightarrow [\sum_{i=1}^{k}n_i D_i]$ where ${\rm Div}(G)$ and ${\rm Prin}(G)$ are the group of divisors and the group of principal divisors of $G$, respectively, and $[D]$, for a divisor $D$, is its linear equivalence class. We refer to Subsection \ref{poingra_sect}  for more details. \qed

The following are two cases that shed light on the general situation.
 
\begin{itemize}
\item {\bf The case $k=1$:}  If ${\rm deg}(D_1)<0$, then $P_{G}(z_1)=0$ (since the rank of a divisor of negative degree is minus one).   If ${\rm deg}(D_1)>0$, then the rationality of $P_{G}$ follows from the observation that for $n_1>>0$, the Riemann-Roch theorem for graphs  \cite[Theorem 1.12]{BakNor07}   implies that $r_G(n_1 D_1)=n_1{\rm deg}(D_1)-g$ where $g$ is the genus of the graph. If ${\rm deg}(D_1)=0$, then  $P_{G}(z_1)=\sum_{n_1 \in {\rm ker}(\phi_{G}) \cap \mathbb{N}} z_1^{n_1}$ which in turn is rational since ${\rm ker}(\phi_{G})$ is a subgroup of $\mathbb{Z}$. In fact, $P_{G}$ is of the form $1/(1-z_1^c)$ for some positive integer $c$ (in fact, $c$ is the order of $[D_1]$ in ${\rm Jac}(G)$). 

\item {\bf The case ${\rm deg}(D_i)=0$ for all $i$:}  The image of $\phi_{G}$ is finite (thanks to the finiteness of the Jacobian group of $G$) and hence, ${\rm ker}(\phi_{G})$ is a finite index sublattice of $\mathbb{Z}^k$. The Poincar\'e series $P_{G}$ is equal to $\sum_{(n_1,\dots,n_k) \in {\rm ker}(\phi_G)\cap \mathbb{N}^k} z_1^{n_1}\cdots z_k^{n_k}$ and hence, is the lattice point enumerating function (with respect to the lattice ${\rm ker}(\phi_{G})$) of the non-negative orthant cone.

\end{itemize}

\subsection{Poincar\'e Series of Divisors on a Tropical Curve}

In the following, we formulate a notion of Poincar\'e series of a finite collection of divisors on an abstract tropical curve. Recall that an abstract tropical curve, is by definition, a compact metric graph, i.e. a compact metric space where every point has a neighbourhood isometric to a star-shaped set,  we refer to Subsection \ref{div_app} for more details, also see \cite[Subsection 3.3]{MikZha08}, \cite[Section 3]{BakFab11}.   A compact metric graph can be represented by a finite graph with edge set $E$ along with a function $\ell:E \rightarrow \mathbb{R}_{\geq 0}$, the function $\ell$ can be interpreted as an assignment of lengths to the edges. 

Abstract tropical curves share various properties with smooth, proper algebraic curves. For instance, they satisfy an analogue of the Riemann-Roch theorem, have an associated Jacobian group and a corresponding Abel Jacobi map  \cite{GatKer08,MikZha08,BakFab11}.  In a related context, compact metric graphs occur as skeleta of the Berkovich analytification of a smooth, proper algebraic curve over a non-archimedean field \cite{BakPayRab16}. In the following, we simply use the term ``tropical curves'' to refer to abstract tropical curves. 

 Given a finite sequence of divisors $D_1,\dots,D_k$ on the tropical curve $\Gamma$.   The Poincar\'e series associated to $D_1,\dots,D_k$ is defined as: 

 \begin{center}

$P_{\Gamma}(z_1,\dots,z_k)=\sum_{(n_1,\dots,n_k) \in \mathbb{N}^k}(r_{\Gamma}(n_1D_1+\dots+n_kD_k)+1)z_1^{n_1} \cdots z_k^{n_k}$

\end{center}

where $r_\Gamma(D)$ is the rank of the divisor $D$ on $\Gamma$. As in the case of graphs, the inequality that the rank of a divisor is at most its degree for every divisor of non-negative degree implies that $P_{\Gamma}$ is absolutely convergent in the region $\{ (z_1,\dots,z_k)|~|z_i|<1 \text{ for all } i \}$.

Next, we consider the problem of rationality of $P_{\Gamma}$. We start by noting some key differences between the case of finite graphs and tropical curves. 
 The Jacobian of a finite graph is a finite Abelian group but the Jacobian of a tropical curve is (except in genus zero) not a finite group (nor a finitely generated group) but is a real torus of dimension $g$, where $g$ is the genus of $\Gamma$ (the first Betti number of the simplicial complex underlying $\Gamma$) \cite[Section 6]{MikZha08},~\cite[Page 364]{BakFab11}.

Furthermore, consider the group homomorphism
$\phi_{\Gamma}: \mathbb{Z}^k \rightarrow {\rm Div}(\Gamma)/{\rm Prin}(\Gamma)$ defined as follows:  
\begin{center}
$\phi_{\Gamma}(m_1,\dots,m_k)=[\sum_{i=1}^{k}m_iD_i]$
\end{center}

where $[.]$ is the associated linear equivalence class in ${\rm Div}(\Gamma)/{\rm Prin}(\Gamma)$.  Note that ${\rm Div}(\Gamma)/{\rm Prin}(\Gamma)$ is isomorphic to ${\rm Jac}(\Gamma) \oplus \mathbb{Z}$ and is the analogue of the Picard group of an algebraic curve.  The image of $\phi_{\Gamma}$ can be more ``complicated'' than its counterpart for graphs.  For instance, it can be infinite even when each $D_1,\dots,D_k$ has degree zero as in the following example. Suppose that $\Gamma$ is a cycle of unit edge length (this is a tropical curve of genus one, i.e. a tropical elliptic curve) and  its Jacobian group is the unit circle $\mathbb{S}^1$.  Consider the parameterisation $e^{2\pi i \theta}$ (where $\theta \in [0,2\pi)$) for $\mathbb{S}^1$. Let $k=1$ and let $p$ be the point in $\Gamma$ whose image in its Jacobian under the Abel-Jacobi map (with respect to a fixed base point $p_0$) is  the point $e^{2 \pi i \phi}$ for an irrational number $\phi$.  Note that such a point $p$ exists since there is a bijection between $\Gamma$ and its Jacobian,  see \cite{Vig09} for more details. 

We set $D_1=(p)-(p_0)$. Since the point $e^{2 \pi i \phi}$  has infinite order in ${\rm Jac}(\Gamma)$,  the image of $\phi_{\Gamma}$ is infinite. Equivalently, the kernel of $\phi_{\Gamma}$  is trivial.   Furthermore, by Weyl's equidistribution theorem \cite[Pages 11--14]{Kor88}, the image of $\phi_{\Gamma}$ is equidistributed in the Jacobian.  In this case, however, the Poincar\'e series is zero since the rank of every multiple of $D_1$ is minus one.  But, via  a slight modification,  we can construct examples (with $k>1$) where the Poincar\'e series is non-zero and the image of $\phi_{\Gamma}$ (restricted to divisor classes with degree in $[0,2g-2]$) is infinite. For instance, set $k=2$, choose $D_1$ to be a divisor of degree $g$ (the genus of $\Gamma$) and $D_2$ to be a divisor of degree zero such that $[D_2]$ has infinite order in ${\rm Jac}(\Gamma)$.   Hence, unlike in the case of graphs, analysing the fibre over each divisor class with degree in $[0,2g-2]$ in the image of $\phi_{\Gamma}$ does not lead to a proof of rationality. 
 
 \begin{figure}
  \includegraphics[width=12cm]{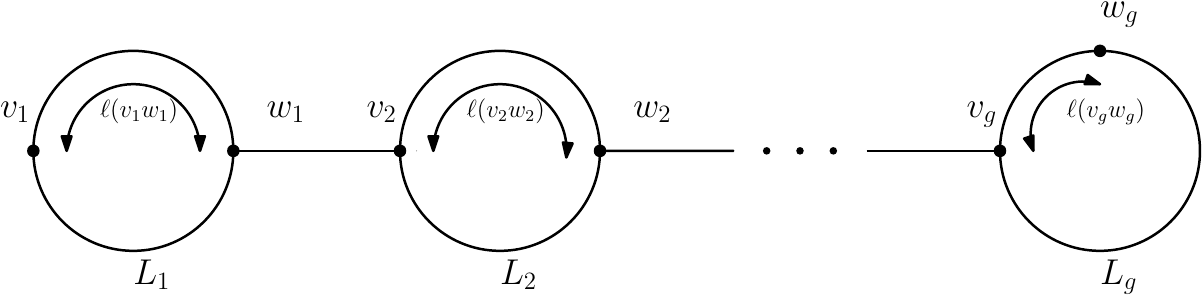}
  \caption{$\Gamma_g$: Chain of Loops of Genus $g$}\label{chainofloops_fig}
\end{figure}

To the best of our knowledge, the problem of rationality of Poincar\'e series of divisors on arbitrary metric graphs is open. In this paper, we study the Poincar\'e series of divisors  on tropical curves whose combinatorial type, i.e. the underlying graph is a chain of loops, see Figure \ref{chainofloops_fig}.  Chains of loops is a well-studied family of tropical curves and has found several applications so far.  For instance,  as an ingredient in the proof that the moduli space of curves of genus 22 and 23 are of general type \cite{FarJenPay20},  a Brill-Noether theory for algebraic curves with fixed gonality \cite{JenRan17},  a proof of the maximal rank conjecture for quadrics \cite{JenPay16} and a proof of the non-existence part of the Brill-Noether theorem for algebraic curves \cite{CooDraPayRob12}.  We show the following rationality result for the Poincar\'e series of divisors on chains of loops.

\begin{theorem}\label{poinratcol_theo} {\rm ({\bf Rationality of Poincar\'e Series of Divisors on Chains of Loops})}
Fix non-negative integers $g$ and $k$. Let $\Gamma_g$ be a chain of loops of genus $g$.  For any finite collection of divisors $D_1,\dots,D_k$ on $\Gamma_g$, the Poincar\'e series $P_{\Gamma_g}(z_1,\dots,z_k)$ is rational.  
\end{theorem} 

{\bf Proof Outline:} We adapt a strategy analogous to Cutkosky's proof of rationality of the corresponding Poincar\'e series for smooth algebraic curves over an algebraically closed field $\mathbb{K}$ of characteristic zero \cite{Cut03}. The main idea behind this proof is to apply Lang's conjecture for subvarieties of (semi-)Abelian varieties, proved by McQuillan \cite{McQ95},\cite[Subsection F.1.1]{HinSil00}, to the Brill-Noether loci of $C$, where $C$ is the underlying algebraic curve. Suppose that $D'_1,\dots,D'_k \in {\rm Div}(C)$ where  ${\rm Div}(C)$ is the group of divisors on $C$.  Fix a point $p_0 \in C$. For integers $r$ and $d$, recall that the Brill-Noether locus $W^r_d(C)$ (with respect to $p_0$) is defined as follows \footnote{This definition is a variant of the more standard definition: $W^r_d(C)=\{ [D] \in {\rm Pic}^d(C)|~r_C(D) \geq r\}$. Note that our definition ensures that $W^r_d(C)$ is a subset of the Jacobian and is related to the standard definition by a translation by $d \cdot (p_0)$. We employ an analogous notion also in the tropical case. }: 

\begin{center} $W^r_d(C)=\{ [D'] \in {\rm Jac}(C)|~r_C(D'+d \cdot (p_0)) \geq r\}$. \end{center}

Note that $r_C(D'+d \cdot (p_0))$ does not depend on the choice of representative in $[D']$.   Let $\bar{D'_i}=D'_i-d'_i \cdot (p_0)$, where $d'_i$ is the degree of $D'_i$.

 Consider the homomorphism  $\bar{\phi}_{C}: \mathbb{Z}^k \rightarrow {\rm Jac}(C)$ given by 
 
 \begin{center} $\bar{\phi}_{C}(m_1,\dots,m_k)=[\sum_{i=1}^{k}m_i \cdot \bar{D'_i}]$.\end{center}

The image of $\bar{\phi}_{C}$ is a finitely generated subgroup $\mathcal{H}$ of ${\rm Jac}(C)$.  The Brill-Noether locus $W^r_d(C)$ is a (closed) subvariety of the Jacobian \cite[Pages 107--152]{ArbCorGriHar85}. By Lang's conjecture, there exists a finite collection of  Abelian subvarieties $\mathcal{A}_1,\dots,\mathcal{A}_s$ of ${\rm Jac}(C)$ and corresponding translates $\gamma_1,\dots,\gamma_s \in \mathcal{H}$  such that the following holds: 

\begin{enumerate}

\item $\gamma_i+\mathcal{A}_i(\mathbb{K}) \subseteq W^r_d(C)$ for each $i$. 

\item $W^r_d(C) \cap \mathcal{H}=\cup_{i=1}^{s}(\gamma_i+(\mathcal{A}_i(\mathbb{K}) \cap \mathcal{H}))$.

\end{enumerate}

The rationality then follows from considering the fibre of $\bar{\phi}_{C}$ over each coset $\gamma_i+(\mathcal{A}_i(\mathbb{K}) \cap \mathcal{H})$.  

Taking cue from this, we study the intersection of the tropical Brill-Noether locus with the subgroup generated by $[\bar{D}_1],\dots,[\bar{D}_k] \in {\rm Jac}(\Gamma_g)$ where for each $i$, $\bar{D}_i=D_i-d_i \cdot (w_g)$ and $d_i$ is the degree of $D_i$.  Our main technical ingredient, developed in Section \ref{LangBNCol_sect}, is an analogue of Lang's conjecture for Brill-Noether loci on chains of loops. The rationality of $P_{\Gamma_g}$  then follows analogous to the case of both algebraic curves  and graphs. We refer to Section \ref{poinratcol_sect} for more details. \qed

{\bf Lang's Conjecture for Brill-Noether Loci on Chains of Loops:}  In the following, we state Lang's conjecture for Brill-Noether loci on chains of loops.  Recall that given integers $r$ and $d$, the Brill-Noether locus $W^r_d(\Gamma_g) \subseteq {\rm Jac}(\Gamma_g)$ (with respect to the fixed point $w_g \in \Gamma_g$) \cite[Subsection 1.2]{Pfl17} is defined as follows:  
 \begin{center} $W^r_d(\Gamma_g)=\{ [D] \in {\rm Jac}(\Gamma_g)|~r_{\Gamma_g}(D+d \cdot (w_g)) \geq r\}$. \end{center}

\begin{theorem}{\rm ({\bf  Lang's Conjecture for Brill-Noether Loci on Chains of Loops})}\label{langbncol_theo}
Let $H$ be a subgroup of ${\rm Jac}(\Gamma_g)$. Suppose that $r,~d$ are integers such that $W^r_d(\Gamma_g) \cap H \neq \emptyset$.  There is a finite collection of tropical Abelian subvarieties $A_1,\dots,A_s$ of ${\rm Jac}(\Gamma_g)$  and translates $\gamma_1,\dots,\gamma_s \in H$ such that the following holds:

\begin{itemize}

\item $\gamma_i+A_i \subseteq W^r_d(\Gamma_g)$ for each $i$ from one to $s$.

\item $W^r_d(\Gamma_g) \cap H=\cup_{i=1}^{s} (\gamma_i+(A_i \cap H))$.

\end{itemize}
\end{theorem}

The proof of Theorem \ref{langbncol_theo} builds on a theorem of Pflueger \cite[Theorem 1.4]{Pfl17}  that states that $W^r_d(\Gamma_g)$ (when non-empty) is a finite union of certain subtori of ${\rm Jac}(\Gamma_g)$.  We refer to Section \ref{LangBNCol_sect} for more details.

\subsection{Algorithmic Aspects}

We present algorithms for computing the Poincar\'e series in both the cases: graphs (Subsection \ref{poingraphexp_sect}) and chains of loops (Section \ref{poicolalg_sect}). The algorithms build on the corresponding rationality proofs with additional effort towards making each step computationally amenable.  We summarise some key steps in the following. 

\begin{itemize}

\item In both the cases, an explicit description of the sum ``corresponding to divisors of degree greater than $2g-2$'' is involved, where $g$ is the genus of the underlying object. In particular, we provide an explicit description of the polyhedron underlying this sum.  We refer to Subsection \ref{poinlarge_subsect} for more details. 

\item In the case of graphs, a computation of the fibres of the homomorphism $\phi_{G}$ is another key ingredient of the algorithm (Subsection \ref{fibrecom_subsect}).

\item In the case of chains of loops, a computation of the fibres of the homomorphism $\bar{\phi}_{\Gamma_g}$ (an analogue of $\bar{\phi}_C$) over the Brill-Noether loci of $\Gamma_g$  (with an additional ``degree'' constraint) is another important ingredient of the algorithm. We refer to Subsection \ref{afflatcon_subsect} for more details.

\end{itemize}

{\bf Future Work:}  The current work initiates the study of Poincar\'e series of divisors on graphs and tropical curves.  A natural next step is to investigate the rationality of Poincar\'e series associated to arbitrary tropical curves. Other interesting directions include their classification and investigating the information they carry about the underlying graph or tropical curve.  A closely related direction is the study of moduli spaces of tropical curves (along with a finite collection of divisors) with a given Poincar\'e series. 

{\bf Acknowledgement:} We thank Steven Dale Cutkosky for several fruitful discussions on this topic and for his comments on an earlier draft. We thank Ye Luo for interesting discussions on this topic.  A part of this work was carried out while we were visiting the International Centre for Theoretical Sciences (ICTS), Bangalore. We thank ICTS for its kind hospitality. 

\section{Preliminaries}
In this section, we touch upon the main objects involved in this paper with the goal of keeping the exposition self-contained. 
\subsection{Divisor Theory on Graphs and Tropical Curves}\label{div_app}

 Let $G$ be a finite, connected, multigraph with set of vertices $V(G)$ and set of edges $E(G)$.  Following Baker and Faber \cite[Section 3]{BakFab11},  a metric graph  $\Gamma$ (also known as an abstract tropical curve) is a compact, connected metric space in which every point $p \in \Gamma$ has a neighbourhood isometric to a star-shaped set with (integer) valence $n_p \geq 1$. A star-shaped set is a set of the form:
\begin{center}
$S(n_p,r_p)=\{z \in \mathbb{C}|~z=t \cdot e^{2 k \pi i/n_p},~t \in [0,r_p],~ k \in \mathbb{Z}\}$
\end{center}

for an integer $n_p \geq 1$ and a non-negative real number $r_p$. Additionally, the set $S(n_p,r_p)$ is equipped with the path metric.  Note that the compactness of $\Gamma$ implies that it only has a finite number of points with valence not equal to two. Any finite, connected, multigraph $G$ with a length function $\ell: E(G) \rightarrow \mathbb{R}_{\geq 0}$ defines a metric graph and conversely, given any metric graph $\Gamma$ there exists a finite, connected, multigraph $G$ with a length function whose associated metric graph is $\Gamma$, called a \emph{model of $\Gamma$} .

Consider an isometry from a neighbourhood of a point $p \in \Gamma$ to a star-shaped set $S(n_p,r_p)$. A tangent at $p$ is the preimage, under such an isometry, of the segment $\{t \cdot e^{2 k \pi i/n_p}|~t \in [0,r_p]\}$ in $S(n_p,r_p)$  for a fixed integer $k$.  We denote by ${\rm Tan}_{\Gamma}(p)$, the set of equivalence classes of all tangents at $p$ where two tangents are equivalent if one is contained in the other and refer to its elements as \emph{tangent directions}. Note that this set does not depend on the choice of neighbourhood.



 Let ${\rm Div}(G)$ be the free Abelian group generated by the vertices of $G$ and let ${\rm Div}(\Gamma)$ be the free Abelian group generated by the points of $\Gamma$.  A divisor on $G$ (and on $\Gamma$) is an element in ${\rm Div}(G)$ (and ${\rm Div}(\Gamma)$), respectively.  We denote a divisor on $G$ by $\sum_{u \in V(G)}a_u (u)$ where $a_u$ is an integer and a divisor on $\Gamma$ by $\sum_{p \in \Gamma}a_p (p)$ where each $a_p$ is an integer and is zero for all but finitely many points of $\Gamma$.
 Both groups are naturally equipped with homomorphisms, namely ${\rm deg}: {\rm Div}(G) \rightarrow \mathbb{Z}$ and ${\rm deg}: {\rm Div}(\Gamma) \rightarrow \mathbb{Z}$ that takes $\sum_{u \in V(G)}a_u (u)$ to $\sum_{u \in V(G)} a_u$ and 
 $\sum_{p \in \Gamma}a_p (p)$ to $\sum_{p \in \Gamma} a_p$ respectively.  The image of a divisor under such a homomorphism is called its \emph{degree}.  A divisor is called \emph{effective} if every coefficient is non-negative.
 
Graphs and abstract tropical curves have a divisor theory akin to the divisor theory on an algebraic curve, we refer to \cite{BakNor07,MikZha08, GatKer08} for a detailed treatment of this topic.  In the following, we briefly recall the notions of rational functions, principal divisors, degree and rank.  A rational function on $G$ is a function $f_G:V(G) \rightarrow \mathbb{Z}$.  The principal divisor ${\rm div}(f_G)$ associated to $f_G$ is defined as  ${\rm div}(f_G)=\sum_{u \in V(G)}a_u(u)$  where $a_u=\sum_{e=(u,v) \in E(G)} (f_G(u)-f_G(v))$.  

 A rational function on $\Gamma$ is a real-valued, continuous, piecewise linear function $f_{\Gamma}: \Gamma \rightarrow \mathbb{R}$ with integer slopes (and finitely many pieces).  For a tangent direction $e \in {\rm Tan}_{\Gamma}(p)$, let ${\rm slp}_e(f_{\Gamma})$ be the outgoing slope of $f_{\Gamma}$ along $e$, i.e. $(f_{\Gamma}(q)-f_{\Gamma}(p))/\ell_t$ where $t$ is a tangent in the equivalence class of $e$, $q$ is the other end point of $t$ and $\ell_t$ is the length of $t$. Note that ${\rm slp}_e(f_{\Gamma})$ does not depend on the choice of $t$.  The principal divisor associated to $f_{\Gamma}$ is defined as  ${\rm div}(f_{\Gamma})=\sum_{p \in \Gamma}a_p(p)$ where $a_p=\sum_{e \in {\rm Tan}_{\Gamma}(p)}{\rm slp}_e(f_{\Gamma})$. Note that since $\Gamma$ only has finitely many points with valence not equal to two and $f_{\Gamma}$ only has finitely many pieces, we know that $a_p=0$ for all but finitely many points $p \in \Gamma$. 

In both cases, the set of principal divisors form a subgroup of the corresponding group of divisors. We denote that by ${\rm Prin}(G)$ and ${\rm Prin}(\Gamma)$ respectively.  Moreover,  ${\rm Prin}(G)$ is a subgroup of ${\rm Div}^{0}(G)$ and ${\rm Prin}(\Gamma)$  is a subgroup of ${\rm Div}^{0}(\Gamma)$ where ${\rm Div}^{0}(G)$  and ${\rm Div}^{0}(\Gamma)$ are the groups of divisors of degree zero on $G$ and $\Gamma$ respectively.

Let $D_1$ and $D_2$ be divisors both on $G$ or both on $\Gamma$. They  are said to be \emph{linearly equivalent} if $D_1-D_2$ is a principal divisor. 
Given a divisor $D$, its linear system $|D|$ is the set of all effective divisors linear equivalent to $D$. The rank $r_{G}(D)$ (or $r_{\Gamma}(D)$) of a divisor $D$ on $G$ (or $\Gamma$ respectively) is minus one if $|D|=\emptyset$ and otherwise, it is the maximum integer $r$ such that $|D-E| \neq \emptyset$ for every effective divisor of degree $r$.

{\bf Jacobians of Graphs and Tropical Curves:} We now briefly discuss the notion of Jacobian. The Jacobian group ${\rm Jac}(G)$ of a graph $G$ is defined as ${\rm Div}^{0}(G)/{\rm Prin}(G)$. Analogously, the Jacobian group ${\rm Jac}(\Gamma)$ of an abstract tropical curve $\Gamma$ is defined as ${\rm Div}^{0}(\Gamma)/{\rm Prin}(\Gamma)$. 

The Jacobian of $G$ is a finite group of order equal to the number of spanning trees of $G$. The relation between its structure and the underlying graph still remains elusive, we refer to  \cite{Woo17} for recent progress on this topic.
 
 The Jacobian of $\Gamma$ is isomorphic to $H_1(G_{\Gamma},\mathbb{R})/H_1(G_{\Gamma},\mathbb{Z})$ for any model $G_{\Gamma}$ of $\Gamma$ \cite[Theorem 6.2]{MikZha08}, \cite[Theorem 2.8]{BakFab11}. Furthermore, it is a real torus of dimension equal to the genus of $\Gamma$, where by genus we mean the first Betti number of any graph underlying a model of $\Gamma$.  
 

\subsection{Tropical Abelian Varieties} \label{tropabel_app}

A  \emph{principally polarised tropical Abelian variety} of dimension $g$ is a pair $(V/\Lambda,Q)$ where $V$ is a real vector space of dimension $g$, $\Lambda$ is a full rank sublattice of $V$ and $Q$ is a  symmetric, positive semidefinite quadratic form on $V$.
Furthermore,  the null space of $Q$ is rational with respect to $\Lambda$, i.e. it has a vector space basis consisting of elements in $\Lambda$ \cite[Section 5]{MikZha08}, \cite[Section 5]{BolBranChu17}.  The quadratic form $Q$ is called a \emph{principal polarisation} of $V/\Lambda$.
Since we only deal with principally polarised tropical Abelian varieties, in the following we simply refer to them as tropical Abelian varieties. Two tropical Abelian varieties $(V_1/\Lambda_1,Q_1)$ and $(V_2/\Lambda_2,Q_2)$ are isomorphic if there is a vector space isomorphism $\sigma:V_1 \rightarrow V_2$ that restricts to an isomorphism between $\Lambda_1$ and $\Lambda_2$ and satisfies $Q_1({\bf p})=Q_2(\sigma({\bf p}))$ for all ${\bf p} \in \Lambda_1$ \cite[Section 4A]{Cha12}.

The Jacobian of a metric graph  $\Gamma$ naturally carries the structure of a tropical Abelian variety that we now describe. Following \cite{BakFab11}, we fix a model $G_{\Gamma}$ for $\Gamma$.   The vector space $V$ in this case is $H_1(G_{\Gamma},\mathbb{R})$ and the lattice $\Lambda$ is $H_1(G_{\Gamma},\mathbb{Z})$. The quadratic form $Q_{\Gamma}$ on $H_1(G_{\Gamma},\mathbb{R}) \subset C_1(G_{\Gamma},\mathbb{R})$  is induced by the standard inner product on $C_1(G_{\Gamma},\mathbb{R})$ (with respect to the basis given by the edges of $G_{\Gamma}$ with each edge carrying an orientation), i.e. 

\begin{center}
$
\langle e_i,e_j \rangle=
\begin{cases}
\ell_{e_i}, \text{~if~}e_i=e_j,\\
0,\text{~otherwise}.
\end{cases}
$
\end{center}

where $\ell_{e_i}$ is the length of the edge $e_i$. Hence, \begin{center} $Q_{\Gamma}(\sum_{e \in E(G_{\Gamma})} \alpha_e \cdot e)=\sum_{e \in E(G_{\Gamma})}\alpha_e^2 \cdot \ell_e$ \end{center} where $E(G_{\Gamma})$ is the set of edges of $G_{\Gamma}$ with each edge carrying an orientation. Given a basis for $H_1(G_{\Gamma},\mathbb{Z})$,  the quadratic form $Q_{\Gamma}$ has an associated  $g \times g$ matrix that is called the \emph{period matrix} of $\Gamma$. For instance, consider the model $G_g$ for the chain of loops $\Gamma_g$ induced by $v_1,~w_g$ along with  the branch points. Suppose that $\mathcal{U}_i$ and $\mathcal{V}_i$ are the upper and lower edges of the loop $L_i$ oriented from $v_i$ to $w_i$. The period matrix for the chain of loops $\Gamma_g$ with respect to the basis $\{ \mathcal{U}_1-\mathcal{V}_1,\dots,  \mathcal{U}_g-\mathcal{V}_g\}$ of $H_1(G_g, \mathbb{Z})$ is a diagonal matrix with the lengths of the loops in the diagonal.


\subsection{Linear Equivalence: Chain of Loops vs One Loop}\label{lineq_app}

Given two divisors both supported in any one loop in a chain of loops. The following proposition relates linear equivalence between them treated as divisors on that loop with linear equivalence between them treated as divisors on the chain of loops.
This comparison is relevant in an algorithm to reduce a divisor on a chain of loops in its linear equivalence class,  we refer to Section \ref{LangBNCol_sect} for more details.

\begin{proposition}\label{lineqcol_lem}
Let $\Gamma_g$ be a chain of $g$ loops for an integer $g \geq 1$.  Fix an integer $j$ between one and $g$. Divisors $D_1$ and $D_2$ both supported on the loop $L_j$ are linearly equivalent as divisors on $L_j$ if and only they are linearly equivalent as divisors on $\Gamma_g$. 
\end{proposition}

\begin{proof} 
($\Rightarrow$) Since $D_1$ and $D_2$ are linearly equivalent as divisors on $L_j$, there is a rational function $f_{L_j}$ on $L_j$ whose principal divisor is $D_1-D_2$. We can extend $f_{L_j}$ to a function $f_{\Gamma_g}$ on $\Gamma_g$ as follows:

$f_{\Gamma_g}(p)=\\
\begin{cases}
f_{L_j}(p),\text{ if }p \in L_j,\\
f_{L_j}(v_j),\text{ if } p \in L_i \text{ for } i<j \text{ or on the segment joining $v_i$ and $w_{i-1}$ for $i \leq j$},\\
f_{L_j}(w_j),\text{ if } p \in L_i \text{ for } i>j \text{ or on the segment joining $v_{i}$ and $w_{i-1}$ for $i>j$},

 \end{cases}
$

By construction, $f_{\Gamma_g}$ is a rational function on $\Gamma_g$ and the principal divisor associated to it is precisely $D_1-D_2$ (as a divisor on $\Gamma_g$).

($\Leftarrow$) Suppose that $D_1$ and $D_2$ are linearly equivalent as divisors on $\Gamma_g$ and let $f_{\Gamma_g}$ be the rational function whose associated principal divisor is $D_1-D_2$. We claim that $f_{\Gamma_g}|L_j$, i.e. $f_{\Gamma_g}$ restricted to $L_j$, is a rational function on $L_j$ whose principal divisor is $D_1-D_2$. Indeed,  $f_{\Gamma_g}|L_j$ is a rational function on $L_j$. In order to show that its principal divisor is $D_1-D_2$, we need to show that $f_{\Gamma_g}$ is locally constant at $v_j$ and $w_j$ along the tangent directions corresponding to the bridges $(v_j,w_{j-1})$ and $(v_{j+1},w_j)$, respectively (whenever they exist).  

To see this,  consider the restriction of $f_{\Gamma_g}$ to the (sub-)metric graph $\mathcal{C} \cup \{v_j\}$ where $\mathcal{C}$ is the connected component of $\Gamma_g \setminus \{v_j\}$ that contains $w_{j-1}$. This restriction $f_{\Gamma_g}|(\mathcal{C} \cup \{v_j\})$ is a rational function on $\mathcal{C} \cup \{v_j\}$. The principal divisor associated to it has degree zero and cannot have any point (in $\mathcal{C} \cup \{v_j\}$) other than $v_j$ in its support (since this property holds for $f_{\Gamma_g}$). Hence, this is the divisor zero. Hence, $f_{\Gamma_g}|(\mathcal{C} \cup \{v_j\})$ and $f_{\Gamma_g}$ are locally constant along the tangent direction corresponding to  $(v_j,w_{j-1})$. Analogously, the fact that $f_{\Gamma_g}$ is locally constant along the tangent direction corresponding to  $(v_{j+1},w_j)$ follows by considering the restriction of $f_{\Gamma_g}$ on $\mathcal{C}' \cup \{w_j\}$ where $\mathcal{C}'$ is the connected component of $\Gamma_g \setminus \{w_j\}$ containing $v_{j+1}$.
\end{proof}

We refer to \cite[Lemma 3.13]{Pfl17} that is closely related to Proposition \ref{lineqcol_lem}.

\section{Poincar\'e Series of Divisors on a Finite Graph}\label{poingra_sect}
We start this section with a proof of rationality of Poincar\'e series of divisors on a finite graph and then discuss its algorithmic aspects. 
We will use the convention that if a series is indexed over an empty set, then it is zero. 

\subsection{Proof of Theorem \ref{grapoinrat_theo}}

Let $d_i={\rm deg}(D_i)$ where ${\rm deg}(.)$ is the degree of the divisor.  We decompose the Poincar\'e series based on the degree of $\sum_{i=1}^{k}n_iD_i$ as follows. 
Let $Q^{(l)}_{G}=\{ (n_1,\dots,n_k)  \in \mathbb{N}^k|~\sum_{i=1}^{k} n_id_i=l\}$, we define

\begin{center} 
$P^{(l)}_{G}(z_1,\dots,z_k)=\sum_{(n_1,\dots,n_k )\in Q^{(l)}_{G}}(r_G(n_1D_1+\dots+n_kD_k)+1)z_1^{n_1} \cdots z_k^{n_k}$.
\end{center} 

Note that the degree of $\sum_{i=1}^{k}n_iD_i$ is $\sum_{i=1}^{k}n_id_i$. 
By construction, $P_{G}=\sum_{l \in \mathbb{Z}} P^{(l)}_{G}$.  Note that $P^{(l)}_{G}=0$ for $l<0$ since the rank of a divisor of negative degree is $-1$.  Furthermore, if $\sum_{i=1}^{k}n_id_i >2g-2$, then by the Riemann-Roch theorem for graphs  \cite[Theorem 1.12]{BakNor07}  $r_G(n_1D_1+\dots+n_kD_k)=\sum_{i=1}^{k} n_i d_i-g$.  Hence, 
\begin{center}
$\sum _{l>2g-2}P^{(l)}_{G}=\sum_{(n_1,\dots,n_k) \in \mathbb{N}^k,~\sum_{i=1}^{k}n_id_i \geq 2g-1} (\sum_{i=1}^{k}n_id_i-g+1) z_1^{n_1}\cdots z_k^{n_k}$.
\end{center}

The rationality of this power series follows from the rationality of lattice point enumerating functions of rational polyhedra. We provide a more explicit description of this rational function in Subsection \ref{poinlarge_subsect}.   

Next, we consider $P^{(l)}_{G}$ for $l$ from $0$ to $2g-2$. We further decompose $P^{(l)}_{G}$ in terms of its divisor classes. For a divisor class $[D] \in {\rm Div}(G)/{\rm Prin}(G)$,  let 
$Q^{[D]}_{G}=\{ (n_1,\dots,n_k)  \in \mathbb{N}^k|~\sum_{i=1}^{k} n_iD_i  \in [D]\}$ and define

\begin{center} $P^{[D]}_{G}(z_1,\dots,z_k)=(r_G(D)+1)\sum_{(n_1,\dots,n_k) \in Q^{[D]}_{G}}z_1^{n_1}\cdots z_k^{n_k}$. \end{center}  
 
Note that the rank $r_G(D)$ does not depend on the choice of representative in the linear equivalence class $[D]$.   Let ${\rm Jac}^{(l)}(G)$ be the set of all linear equivalence classes of divisors of degree $l$, we have \begin{equation}\label{poindec_eq} P^{(l)}_{G}=\sum_{[D] \in {\rm Jac}^{(l)}(G)} P^{[D]}_{G}.\end{equation}

Note that since ${\rm Jac}^{(l)}(G)$ is a finite set (of cardinality equal to the number of spanning trees of $G$), the sum on the right hand side of Equation (\ref{poindec_eq}) is a finite sum. Hence, it suffices to show that each $P^{[D]}_{G}$ is rational.  


For this, consider the generating function $f(Q^{[D]}_{G};z_1,\dots,z_k)$ defined as \begin{center}$f(Q^{[D]}_{G};z_1,\dots,z_k)=\sum_{(n_1,\dots,n_k) \in Q^{[D]}_{G}} z_1^{n_1}\cdots z_k^{n_k}$. \end{center}  Note that $P^{[D]}_{G}(z_1,\dots,z_k)=(r_G(D)+1)f(Q^{[D]}_{G}; z_1,\dots,z_k)$.  

Next, we study the set $Q^{[D]}_{G}$ in more detail.  Consider the group homomorphism $\phi_{G}: \mathbb{Z}^k \rightarrow  {\rm Div}(G)/{\rm Prin}(G)$ defined as $(n_1,\dots,n_k) \rightarrow [\sum_{i=1}^{k}n_i D_i]$.  The set $Q^{[D]}_{G}$ is then the set of points of the fibre of $\phi_{G}$ over $[D]$ that lie in the non-negative orthant cone.  Since  $\phi_{G}$ is a group homomorphism, the non-empty fibres are cosets of its kernel.  Furthermore, the kernel of $\phi_{G}$ is a sublattice of $\mathbb{Z}^k$ (see Subsection \ref{fibrecom_subsect} for more details). Hence, each non-empty fibre $F_{[D]}$ of $\phi_{G}$ over $[D] \in  {\rm Div}(G)/{\rm Prin}(G)$ is an affine lattice of the form ${\bf a}+{\rm ker}(\phi_{G})$ where ${\bf a} \in \mathbb{Z}^k$  and ${\rm ker}(\phi_{G})$ is the kernel of $\phi_{G}$. The set $Q^{[D]}_{G}$ is the set of points in $F_{[D]}$ that lie in the non-negative orthant cone.   The rationality of $f(Q^{[D]}_{G}; z_1,\dots,z_k)$ follows from \cite[Corollary 7.6]{CutHerReg03}. This completes the proof of rationality of $P_G$. The property that this rational function agrees with the corresponding power series at every point where the power series is absolutely convergent follows from the corresponding property for each lattice point enumerating function in the sum. 
\qed

\subsection{Algorithmic Aspects}\label{poingraphexp_sect}

In this subsection, we describe the rational function associated to the Poincar\'e series of divisors on a finite graph more explicitly with the following two goals in mind: i. to obtain an effective method to construct the rational function given the graph and the divisors, ii. to extract information about the underlying graph from the Poincar\'e series associated to divisors on it.   We start with the summand  $\sum_{l>2g-2} P^{(l)}_{G}$.

\subsubsection{An Explicit Description of $\sum_{l>2g-2} P^{(l)}_{G}$}\label{poinlarge_subsect}

Recall from the last section that $\sum_{l>2g-2} P^{(l)}_{G}(z_1,\dots,z_k)$ is (as a formal power series) equal to:

\begin{center}

 $\sum_{(n_1,\dots,n_k) \in \mathbb{N}^k,~\sum_{i=1}^{k}n_id_i \geq 2g-1} (\sum_{i=1}^{k}n_id_i-g+1) z_1^{n_1}\cdots z_k^{n_k}$.
 
 \end{center}
 
 Consider the lattice point enumerating function $f(Q;z_1,\dots,z_k)$ of the rational polyhedron $Q$ obtained by intersecting the non-negative orthant cone and the half-space $\sum_{i=1}^{k}n_id_i \geq 2g-1$, i.e.
 
 \begin{center}
 $Q=\{(n_1,\dots,n_k) \in \mathbb{R}^k|~n_i \geq 0 \text{ for all }  i, \sum_{i=1}^{k}n_id_i \geq 2g-1\}$.
\end{center}

We can express $\sum_{l>2g-2} P^{(l)}_{G}$ in terms of $f(Q;z_1,\dots,z_k)$ as:

\begin{center}

$(\sum_{i=1}^{k} d_i\partial_{z_i} -(g-1))f(Q;z_1,\dots,z_k)$
\end{center}

where $\partial_{z_i}$ is the partial derivative operator with respect to $z_i$. To compute $f(Q;z_1,\dots,z_k)$, we use Brion's formula \cite{Bri88},\cite[Theorem 3.5]{Bar99}: 

\begin{theorem}
Let $R$ be a rational polyhedron with vertex set $V(R)$. Let ${\rm cone}(v)$ be the tangent cone of the vertex $v$. The lattice point enumerating function of $R$ is given by the formula: 
\begin{center} 
$f(R;z_1,\dots,z_k)=\sum_{v \in V(R)} f({\rm cone}(v);z_1,\dots,z_k)$
\end{center}

where $f({\rm cone}(v);z_1,\dots,z_k)$ is the lattice point enumerating function of ${\rm cone}(v)$.

\end{theorem} 

In the following proposition, we describe the set of vertices of $Q$ and their respective tangent cones.  For an integer $1 \leq  i \leq k$ such that $d_i \neq 0$, we define the point $v_i \in \mathbb{R}^k$ as follows:

\begin{center}
$(v_i)_j=
\begin{cases}
(2g-1)/d_i,~\text{if } j=i,\\
0,~\text{otherwise}. \\
\end{cases}
$

\end{center}
where $(v_i)_j$ is the $j$-th coordinate of $v_i$.   For $i \neq j$, let $R_{i,j}$ be the ray spanned by the vector $v_j-v_i$.

\begin{proposition}
If $g=0$, the vertices of $Q$ are precisely the origin and the points $v_i$ for which $d_i<0$. If $g \geq 1$, the set $Q$ is empty if all $d_i \leq 0$ and otherwise, its vertices are precisely the points $v_i$ for which $d_i>0$.  

The extremal rays of the tangent cone of the origin (in the case $g=0$) are precisely the standard basis vectors $e_1,\dots,e_k$.  One extremal ray of the tangent cone of the vertex $v_i$ is $-e_i$, if $g=0$ and $e_i$, otherwise.  In both the cases, the other extremal rays of the tangent cone of the vertex $v_i$ are $e_j$ such that $j \neq i$ and $d_j=0$, and $R_{i,j}$ for $j$ between $1$ and $k$ such that $d_j \neq 0$ and $j \neq i$.  
\end{proposition}
\begin{proof}
The polyhedron $Q$ can be expressed as the feasible set of the following $k+1$ linear inequalities: 
\begin{center}
$x_i \geq 0$, for all integers $i \in [1,k]$,\\
$\sum_{i=1}^k d_ix_i \geq 2g-1$.
\end{center}
Vertices of $Q$ are precisely those points that attain an equality at $k$ linearly independent constraints (i.e. the corresponding linear system obtained by replacing the inequalities by equalities has full rank) and satisfy the other inequality. The description of the vertices follows immediately from this property. 

The tangent cone of a vertex $v$ is defined by precisely the $k$ constraints that are active at $v$ (since $g$ is an integer, not all the $k+1$ constraints can be active at $v$). Its extremal rays are defined by the equalities corresponding to any $k-1$ of these constraints and the inequality corresponding to the other one. 
The statement on the extremal rays of the tangent cones follows immediately from this observation. 
\end{proof}
\begin{remark}
\rm Note that $Q$ can either be empty (if $d_i<0$ for all integers $i \in [1,k]$ and $g>0$), a bounded polyhedron, i.e. a polytope (if $d_i<0$ for all integers $i \in [1,k]$ and $g=0$) or an unbounded polyhedron (if  $d_i>0$ for all integers $i \in [1,k]$). \qed
\end{remark}

As described in \cite[Proof of Theorem 3.1]{Bar99}, we compute $f({\rm cone}(v);z_1,\dots,z_k)$ as follows. Consider the polyhedron $({\rm cone}(v),1)$ in $\mathbb{R}^{k+1}$ and take the closure $K$ of its conic hull. This is a rational cone (spanned by the generators of ${\rm cone}(v)$ and $e_{k+1}$). The lattice point enumerating function $f(K;z_1,\dots,z_k,t)$ of $K$ can be computed as in \cite[Example 3.3]{Bar99}. We recover $f({\rm cone}(v);z_1,\dots,z_k)$ as $\partial_t f(K;z_1,\dots,z_k,t)|_{t=0}$.
 

\subsubsection{Computing the Fibres of $\phi_{G}$}\label{fibrecom_subsect}

In the following, we compute the fibre of  $\phi_{G}$ over a divisor class $[D]$ of ${\rm Div}(G)/{\rm Prin}(G)$ as a affine sublattice of $\mathbb{Z}^k$, in terms of a translate and a generating set of the underlying sublattice. 
We start by describing the kernel ${\rm ker}(\phi_{G})$ of the homomorphism $\phi_{G}:\mathbb{Z}^k \rightarrow  {\rm Div}(G)/{\rm Prin}(G)$.  We have:
\begin{center} ${\rm ker}(\phi_{G})=\{ (n_1,\dots,n_k) \in \mathbb{Z}^k|~\sum_{i=1}^{k}n_i[D_i] \in {\rm Prin}(G)\}$. \end{center}

By identifying ${\rm Div}(G)$ with the integer lattice $\mathbb{Z}^N$ where $N$ is the number of vertices of the graph,  the group ${\rm Prin}(G)$ of principal divisors on $G$ can be realised a sublattice of $\mathbb{Z}^{N}$ called \emph {the Laplacian lattice $L_G$ of $G$} (the lattice generated by the rows of the Laplacian matrix of $G$) \cite{AmiMan10}. Hence, the problem of computing ${\rm ker}(\phi_{G})$ reduces to

\begin{center} ${\rm ker}(\phi_{G})=\{ (n_1,\dots,n_k) \in \mathbb{Z}^k|~\sum_{i=1}^{k}n_iD_i \in L_G\}$. \end{center}

A finite generating set of ${\rm ker}(\phi_{G})$ can be computed as follows. Compute a basis $\{\mathcal{B}_1,\dots,\mathcal{B}_k\}$ for the lattice $\{(n_1,\dots,n_k) \in \mathbb{Z}^k|~\sum_{i=1}^{k} d_in_i=0 \}$ \footnote{a basis can be computed by first computing a $\mathbb{Q}$-vector space basis $\{\tilde{\mathcal{B}}_1,\dots,\tilde{\mathcal{B}}_k\}$ for $\{(n_1,\dots,n_k) \in \mathbb{Q}^k|~\sum_{i=1}^{k} d_in_i=0 \}$ and enumerating all lattice points in the parallelopiped defined $\chi_1 \tilde{\mathcal{B}_1},\dots, \chi_k\tilde{\mathcal{B}_k}$ where $\chi_i$ is the minimum non-negative integer such that $\chi_i\tilde{\mathcal{B}}_i$ is in the lattice.}.  For each basis element $\mathcal{B}_i$, compute the smallest non-negative integer $\lambda_i$ such $\lambda_i \cdot \mathcal{B}_i \in L_G$.  Note that $\lambda_i$ is also the order of   $\mathcal{B}_i$, seen as a divisor class in the Jacobian group of $G$.  Suppose $\mathcal{P}$ is the convex hull of $\{ \lambda_1  \cdot\mathcal{B}_1.\dots.\lambda_k  \cdot \mathcal{B}_k \}$. The set $\mathcal{P} \cap L_G$ is a finite generating set of ${\rm ker}(\phi_{G})$. This set can be computed using standard lattice point enumeration algorithms \cite[Section 4]{Bar99}. 

 More generally, the fibre $\phi_{G}$ over a divisor class $[D]$ can be computed as follows. Let $\hat{\phi}_{G}:\mathbb{Z}^{k+1} \rightarrow {\rm Div}(G)/{\rm Prin}(G)$ be the homomorphism given by $\hat{\phi}_{G}((n_1,\dots,n_{k+1}))=\sum_{i=1}^{k}n_i[D_i]+n_{k+1}[D]$. Compute a finite generating set $\hat{\mathcal{B}_1},\dots,\hat{\mathcal{B}}_{k+1}$ for the kernel of $\hat{\phi}_{G}$. The fibre  of $\hat{\phi}_{G}$ over $[D]$ is non-empty if and only if the set $\{(\hat{\mathcal{B}}_1)_{k+1},\dots,(\hat{\mathcal{B}}_{k+1})_{k+1}\}$ of the $(k+1)$-st coordinates of the generators is relatively prime. Suppose that this fibre is non-empty, then there exist integers $c_1,\dots,c_{k+1}$ such that $\sum_{i=1}^{k+1}c_i(\hat{\mathcal{B}}_i)_{k+1}=1$. The fibre of $\hat{\phi}_{G}$  over $[D]$ is $\sum_{i=1}^{k+1}c_i \cdot \hat{\mathcal{B}}_i+{\rm ker}(\hat{\phi}_{G})$. 

\subsubsection{The Algorithm}

We present an algorithm to compute the Poincar\'e series of divisors on a graph $G$ given the Laplacian matrix of $G$ and divisors $D_1,\dots,D_k$. The key ingredients of the algorithm are: 
\begin{enumerate}
\item Computing the fibres of $\phi_{G}$, in particular its kernel, as described in Subsection \ref{fibrecom_subsect}. 
\item Enumerating the divisor classes in the Jacobian of $G$. This can be carried out, for instance, by enumerating all $v$-reduced divisors of degree zero for any fixed vertex $v$ of $G$. By the definition \cite[Section 3.1]{BakNor07}) of reduced divisors, this problem can be formulated as a lattice point enumeration problem. An alternative method is to compute the numerator ($K$-polynomial) of the $\mathbb{N}^{n}$-graded Hilbert series of the $G$-parking function ideal of the graph \cite{PosSha04}. 
\item Computing the lattice point enumerating function of the non-negative orthant cone with respect to an affine sublattice of $\mathbb{Z}^k$. This can be carried out by a suitable modification of lattice point enumeration with respect to $\mathbb{Z}^k$, see \cite[Page 113]{Bar99}, \cite[Theorem 7.5]{CutHerReg03} for more details. 
\item Computing the rank of a divisor on a graph, see  \cite[Chapter 5]{Man11}, \cite{CorLeb16} for more details. 
\item Computing the lattice point enumerating function of the polyhedron $Q$ (of Subsection \ref{poinlarge_subsect}) with respect to $\mathbb{Z}^k$.
\end{enumerate}



\begin{algorithm}\label{Poingra_algo}
\caption{An Algorithm to Compute the Poincar\'e Series of Divisors on a Graph.}
\begin{algorithmic}
\State {\bf Input:} The Laplacian matrix of $G$ and a finite collection of divisors $D_1,\dots,D_k$ on $G$. 
\State Compute the genus $g$ of $G$.
\ForEach{$i$ from one to $k$}
\State Compute the degree $d_i$ of $D_i$.
\EndFor
\State Set $L_G$ to be the Laplacian lattice of $G$. 
\State Set $\Lambda:=\{(n_1,\dots,n_k) \in \mathbb{Z}^k|~\sum_{i=1}^{k}n_iD_i\in L_G\}$.
\State Set $P:=0$.
\State Enumerate all the divisor classes in the Jacobian ${\rm Jac}(G)$.
\State Fix a vertex $v$ of $G$.
\ForEach{$[D] \in {\rm Jac}(G)$ and integer $d \in [0,\dots,2g-2]$}
\If{$\sum_{i=1}^{k}n_i[D_i]=[D]+d \cdot [(v)]$ has an integral solution}
 \State Compute such an integral solution ${\bf s}$.
 \State Compute the lattice point enumerating function $f$ of $\mathbb{R}^{k}_{\geq 0}$ with respect to the affine lattice ${\bf s}+\Lambda$.
 \State Compute the rank $r(D+d \cdot (v))$ of $D$.
 \State Compute $P:=P+r(D+d \cdot (v)) \cdot f$.
 \EndIf
\EndFor

\State  Compute the lattice point enumerating function  $f_Q$ (with respect to $\mathbb{Z}^k$) of $Q=\{(n_1,\dots,n_k) \in \mathbb{R}^k|~n_i \geq 0 \text{ for all }  i, \sum_{i=1}^{k}n_id_i \geq 2g-1\}$.
\State Add $(\sum_{i=1}^{k} d_i\partial_{z_i} -(g-1))f_Q$ to $P$.
\State {\bf Output:} $P$.


\end{algorithmic}
\end{algorithm}

\subsubsection{An Example}

In the following, we will describe the case where $k=N-1$ where $N$ is the number of vertices of $G$ and $D_1,\dots,D_{N-1}$ are a (standard) basis of the root lattice $A_{N-1}(:=(1,\dots,1)^{\perp} \cap \mathbb{Z}^{N})$. 

We denote the vertices of $G$ by $v_1,\dots,v_N$ and let $D_i=(v_i)-(v_N)$ for integers $i$ from $1$ to $N-1$. Note that $\{D_1,\dots,D_{N-1}\}$ is a basis of the root lattice $A_{N-1}$ in $\mathbb{Z}^{N}$ (here ${\rm Div}(G)$ has been identified with $\mathbb{Z}^{N}$ by identifying $(v_i)$ with the standard basis element $e_i$ in $\mathbb{Z}^N$). 
 In this case, ${\rm ker}(\phi_{G})$ can be described more explicitly as follows:
 
 \begin{proposition}\label{kerex_prop}
 The kernel ${\rm ker}(\phi_{G})$ of $\phi_{G}$ is the sublattice of $\mathbb{Z}^{N-1}$ generated by $b_1|_{N-1},\dots,b_{N-1}|_{N-1}$ where $b_i \in \mathbb{Z}^{N}$ is the $i$th row of the Laplacian matrix of $G$ and $b_i|_{j}$ is its restriction to its first $j$ coordinates.  The index $[\mathbb{Z}^{N-1}:{\rm ker}(\phi_{G}]$ of ${\rm ker}(\phi_{G})$ in $\mathbb{Z}^{N-1}$ is equal to the number of spanning trees of $G$. 
\end{proposition}

\begin{proof}
Since $\{D_1,\dots,D_{N-1}\}$ is a basis of $A_{N-1}$ and $L_G \subseteq A_{N-1}$ every element in $L_G$ can be written uniquely as their integer linear combination. The first $N-1$ rows $b_1,\dots,b_{N-1}$ of the Laplacian matrix of $G$ form a basis of $L_G$. They can be expressed as an integer linear combination of $D_1,\dots,D_{N-1}$ as $b_i=\sum_{j=1}^{N-1} (b_i|_{N-1})_j D_{j}$ where $(b_i|_{N-1})_j$ is the $j$-th coordinate of $b_i|_{N-1}$. The first part of the proposition follows from this statement.  This combined with the matrix tree theorem yields the second statement.

\end{proof}

Using Proposition \ref{kerex_prop}, we compute the Poincar\'e series of $D_1,\dots,D_{N-1}$ on $G$. Note that since the degree of every divisor is zero and every divisor of degree zero that is not principal has rank minus one, $P_{G}(z_1,\dots,z_{N-1})=P^{[O]}_{G}(z_1,\dots,z_{N-1})$ where $[O]$ is the identity of the Jacobian of $G$.  Furthermore,  $P^{[O]}_{G}(z_1,\dots,z_{N-1})$ is the lattice point enumerating function of the non-negative orthant cone in $\mathbb{R}^{N-1}$ with respect to the lattice ${\rm ker}(\phi_{G})$. This function can be computed using the following description of the lattice point enumerating function of a rational simplicial cone \cite[Example 3.3]{Bar99}.

Suppose that $C$ is a rational simplicial cone with respect to the lattice $L$, i.e. there is a linearly independent generating set of $C$ consisting only of primitive points in $L$.  Assume, without loss of generality, that the dimension of $C$ is equal to the rank of $L$ denoted by $d$, say. 
Suppose that $\{{\bf g_1},\dots,{\bf g_ d}\}$ be such a generating set.  Let $F=\{\sum_{j=1}^{d}\alpha_j{\bf g_j}|~\alpha_j \in [0,1),~{\rm for~all}~\alpha_j \}$ be the (semi-open) fundamental parallelepiped spanned by this generating set. Note that the generators $\{{\bf g_1},\dots,{\bf g_ d}\}$ span a sublattice of $L$ of finite index $q$, say. Hence, $|F \cap L|=q$ and 
 let $\{{\bf r_1},\dots,{\bf r_q}\}$ be the set of points in $F \cap L$.

 \begin{proposition}\label{ratconegen_prop}
 The lattice point enumerating function $\sum_{{\bf p} \in C \cap L} {\bf z}^{\bf p}$ of $C$ with respect to $L$ is given by 
 \begin{center}
 $\sum_{i=1}^{q} {\bf z^{\bf r_i}}/((1-{\bf z^{g_1}})\cdots(1-{\bf z^{g_d}}))$.
 \end{center}
 \end{proposition}

The generators ${\bf g_1},\dots,{\bf g_{N-1}}$ of the non-negative orthant cone (as in Proposition \ref{ratconegen_prop}) are $\mu_i e_i$ where $e_i$ is the standard basis element of $\mathbb{Z}^{N-1}$ and $\mu_i$ is the order of the element $[D_i]$ in the Jacobian of $G$. The lattice spanned by  ${\bf g_1},\dots,{\bf g_{N-1}}$ is a sublattice of ${\rm ker}(\phi_{G})$ of index $(\prod_{i=1}^{N-1} \mu_i)/N_G$ where $N_G$ is the number of spanning trees of $G$ (note that this also implies that $N_G$ divides $\prod_{i=1}^{N-1} \mu_i$).   As a corollary to Proposition \ref{ratconegen_prop} we have:

\begin{corollary}
The Poincar\'e series $P_{G}(z_1,\dots,z_{N-1})$ is given by the rational function: 
\begin{center}
$(\sum_{{\bf r } \in B}{\bf z^r})/(\prod_{i=1}^{N-1} (1-z_i^{\mu_i}))$
\end{center}
where $B=\{{\bf r} \in {\rm ker}(\phi_{G})|$~the $i$-th coordinate $r_i$  of ${\bf r}$ satisfies $0 \leq r_i <\mu_i$ for each $i$ from $1$ to $N-1\}$. 

Furthermore, for the complete graph $K_N$,  the Poincar\'e series simplifies to the following:

\begin{center}
$P_{K_N}(z_1,\dots,z_{N-1})$\\$=(1-(z_1\cdots z_{N-1})^N)/(1-z_1^N) \cdots (1-z_{N-1}^N)(1-z_1\cdots z_{N-1})$.
\end{center}

\end{corollary}
\begin{proof}
The first part is an immediate consequence of Proposition \ref{ratconegen_prop}.  For the second part, we show that $\mu_i=N$ for all integers $i$ from $1$ to $N-1$. For this, note that $D_i=1/N(\Delta_{K_N}({\bf I}_i-{\bf I}_N))$ where $\Delta_{K_N}$ is the Laplacian operator on $K_N$ and for an integer $1 \leq j \leq N$,  the function ${\bf I}_j$ is the indicator at the vertex $j$. Hence, $N \cdot D_i=N(v_i)-N(v_{N})$ is a principal divisor. To see that $N$ is the smallest non-negative integer with this property, suppose that $m \cdot D_i$ is a principal divisor for some integer $ 0< m<N$.  This would contradict the fact that $\Delta_{K_N}({\bf I}_i)$ and $\Delta_{K_N}({\bf I}_N)$ are contained in a basis of the Laplacian lattice $L_{K_N}$ of $K_N$.  Hence, the discriminant of the sublattice formed by $\{\mu_ie_i\}_{i=1}^{N-1}$ is equal to $N^{N-1}$. By Kirchhoff's matrix-tree theorem and Proposition \ref{kerex_prop}, the discriminant of ${\rm ker}(\phi_{K_N)}$ is $N^{N-2}$. 

Hence, the index of the sublattice spanned by $\{\mu_ie_i\}_{i=1}^{N-1}$ in ${\rm ker}(\phi_{K_N})$  is equal to $N$. This implies that the set $B$ contains precisely $N$ points and they are $(0,\dots,0),(1,\dots,1),\dots,(N-1,\dots,N-1)$. The formula for $P_{K_N}(z_1,\dots,z_{N-1})$ follows from the previous statement and the first part of the proposition.
\end{proof}

We leave the problem of obtaining a more explicit description of the Poincar\'e series of $D_1,\dots,D_{N-1}$ for arbitrary graphs for future work. This seems to need a better understanding of the order $\mu_i$ of $[D_i-D_N]$ in the Jacobian group.  In particular, we are not aware of a description of $\mu_i$ in terms of the underlying graph, we refer to \cite{BecGla16} for some related work.

\section{Lang's Conjecture for Brill-Noether Loci on Chains of Loops}\label{LangBNCol_sect}

In this section, we establish an analogue of Lang's Conjecture for Brill-Noether loci on chains of loops (Theorem \ref{langbncol_theo}). 
 In the following, we briefly recall relevant properties of the Jacobian of a chain of loops.  
 
 {\bf Jacobians of Chains of Loops:}  The Jacobian of a chain of loops of genus $g$ is a real torus of dimension $g$.  Pflueger \cite[Lemma 3.3]{Pfl17} showed that each divisor class in ${\rm Jac}(\Gamma_g)$ has a unique representative of the form $\sum_{i=1}^{g} (\xi_i) -g \cdot (w_g)$ where $\xi_i \in L_i$ for each $i$ from $1$ to $g$.   We refer to these representatives as \emph{Pflueger reduced divisors}. The point $\xi_j$ (and its associated divisor $(\xi_j)$) is called the $j$-th component of the Pflueger reduced divisor. The Jacobian is naturally a \emph{principally polarized tropical Abelian variety} \cite{MikZha08, FosRabShoSot18, BolBranChu17}, i.e. ${\rm Jac}(\Gamma_g)=\mathbb{R}^g/\Lambda$ where $\Lambda$ is a full rank sublattice of $\mathbb{R}^g$ and carries a positive semidefinite quadratic form induced by the period matrix of $\Gamma$, see Subsection \ref{tropabel_app} for more details. 


As evident from the statement of Theorem \ref{langbncol_theo}, subtori of the Jacobian play an important role. In the following, we identify two types of subtori. 
 
 {\bf Standard Topological Subtori:} The Jacobian of $\Gamma_g$ contains topological subtori corresponding to subchains of loops that can be described as follows. Given a non-empty subset $S \subseteq [1,\dots,g]$, the $|S|$-dimensional subtorus of ${\rm Jac}(\Gamma_g)$ associated to $S$ is defined as follows: 
\begin{center}
 $T_S=\{[D]|~D=\sum_{j \in S}(\xi_j)-|S|\cdot (w_g),~\xi_j \in L_j \}$  
 \end{center}
The uniqueness of Pflueger reduced divisors in each linear equivalence class implies that two distinct divisors of the form  $\sum_{i \in S}(\xi_i)-|S|\cdot (w_g)$ are not linearly equivalent and this implies that $T_S$ is a topological subtorus of dimension $|S|$.  We refer to this subtorus as the \emph{standard topological (sub)torus} $T_S$ of ${\rm Jac}(\Gamma_g)$ associated to $S$. A standard topological subtorus is, in general, not a subgroup of ${\rm Jac}(\Gamma_g)$. For an example, consider a chain of two loops and let $S=\{L_1\}$. For any pair of elements in $T_S$, their sum is not contained in $T_S$. However, standard topological subtori are cosets of certain subgroup tori that we now describe. 

{\bf Standard Subgroup Tori:}  Note that any divisor of degree one on a (single) loop is linearly equivalent to the divisor associated to a (unique) point \cite[Proof of Lemma 3.3]{Pfl17}. For $j \in [1,\dots,g]$, let $\mathfrak{o}_j$ be the unique point in $L_j$ that is linear equivalent (with respect to $L_j$) to $j \cdot (w_j)-(j-1) \cdot (v_j)$. We define the set $\mathfrak{T}_S$  as follows: 

\begin{center}
$\mathfrak{T}_S=\{[D]|~D=\sum_{j \in S} (\xi_j)+\sum_{j \notin S}(\mathfrak{o}_j)-g \cdot (w_g),~\xi_j \in L_j\}$ 
\end{center}

Note that divisors of the form $\sum_{j \in S} (\xi_j)+\sum_{j \notin S}(\mathfrak{o}_j)-g \cdot (w_g)$ are Pflueger reduced and the uniqueness of Pflueger reduced divisors in each divisor class implies that this set is a topological subtorus  of ${\rm Jac}(\Gamma_g)$ of dimension $|S|$.   As we shall see in  Proposition \ref{subgroup_prop},  the set $\mathfrak{T}_S$ is also a subgroup of ${\rm Jac}(\Gamma_g)$. In the following, we refer to $\mathfrak{T}_S$ as the \emph{standard subgroup torus} associated to $S$. Note that $T_S$ is equal to $t_S+\mathfrak{T}_S$ where $t_S=[-\sum_{j \notin S}(\mathfrak{o}_j)+(g-|S|) \cdot (w_g)]$.  The proof of the subgroup property of $\mathfrak{T}_S$ uses an algorithm to transform an arbitrary divisor to its Pflueger reduced linear equivalent, due to Pflueger \cite[Lemma 3.3]{Pfl17}, that we describe in the following.





\begin{enumerate}
\item Given a divisor $D \in \Gamma_g$,  first ensure that $D$ is supported only on the loops (not on the bridges). This can done, since for any point $p$ on the bridge between the loops $L_j$ and $L_{j+1}$: the divisor $(p)$ is linearly equivalent to both $(w_j)$ and $(v_{j+1})$.

\item Starting from the first loop, for each loop $L_j$ for $j$ from one to $g-1$, add a suitable multiple of $(w_j)-(v_{j+1})$ to $D$ such that the restriction of $D$ to $L_j$ has degree one.  Add and subtract $g \cdot (w_g)$ to the resulting divisor. 
Any divisor of degree one supported on the loop $L_j$ for $j$ from one to $g$  is linearly equivalent (with respect to  both $L_j$ and $\Gamma_g$, see Subsection \ref{lineq_app}) to $(p_j)$ for a (unique) point $p_j \in L_j$, see  \cite[Proof of Lemma 3.3]{Pfl17}.  Output $\sum_{j=1}^{g}(p_j)-g \cdot (w_g)$.  
 \end{enumerate}

\begin{proposition}\label{subgroup_prop}
The subset $\mathfrak{T}_S$ of ${\rm Jac}(\Gamma_g)$ is a subgroup.
\end{proposition}
\begin{proof}
We show that $\mathfrak{T}_S$ is closed under addition (the group operation) and under inverses.

{\bf Closure under addition:} Let $D_1=\sum_{j \in S} (\xi^{(1)}_j)+\sum_{j \notin S}(\mathfrak{o}_j)-g \cdot (w_g)$ and $D_2=\sum_{j \in S} (\xi^{(2)}_j)+\sum_{j \notin S}(\mathfrak{o}_j)-g \cdot (w_g)$. By the reduction algorithm, we note that for every $j \notin S$ the $j$-th component of the Pflueger reduced divisor of $D_1+D_2$ is the unique point in $L_j$ whose associated divisor is linearly equivalent to $-j \cdot (w_j)+(j-1) \cdot (v_j)+2j \cdot (w_j)-2(j-1) \cdot (v_j)=j \cdot (w_j)-(j-1) \cdot(v_j)$. Hence, the $j$-th component of  the Pflueger reduced divisor of $D_1+D_2$ is $\mathfrak{o}_j$ for every $j \notin S$. This implies that $[D_1+D_2] \in \mathfrak{T}_S$.  

{\bf Closure under inverses:} As in the previous case, for every $j \notin S$ we compute the $j$-th component of the Pflueger reduced divisor of $-D_1$ to be the unique point in $L_j$ whose associated divisor is linearly equivalent to $2j \cdot (w_j)-2(j-1) \cdot (v_j)-j \cdot (w_j)+(j-1) \cdot (v_j)=j \cdot (w_j)-(j-1) \cdot (v_j)$. Hence, $[-D_1] \in \mathfrak{T}_S$.  

 \end{proof}

 Furthermore, $\mathfrak{T}_S$ inherits the structure of a tropical Abelian subvariety from ${\rm Jac}(\Gamma_g)$ (see Subsection \ref{tropabel_app} for the corresponding definitions).  More precisely, suppose that ${\rm Jac}(\Gamma_g)=\mathbb{R}^g/\Lambda$ then $\mathfrak{T}_S=V'/\Lambda'$ where $V'$ is a subspace of $\mathbb{R}^g$ and $\Lambda' \subseteq V'$ is a (saturated) sublattice of $\Lambda$ with full rank in $V'$.  The principal polarisation of ${\rm Jac}(\Gamma_g)$ restricted to $V'$  induces a principal polarisation on $\mathfrak{T}_S$. 
 
 In fact, $\mathfrak{T}_S$ as a tropical Abelian variety is isomorphic to the Jacobian of a chain of $|S|$ loops where the $i$-th loop has edge length equal to the $i$-th loop in $S$ (where the loops in $S$ are in increasing order of their index).

\begin{proof}({\bf Proof of Theorem \ref{langbncol_theo}})
Since $W^r_d(\Gamma_g) \cap H \neq \emptyset$, we have $W^r_d(\Gamma_g)  \neq \emptyset$. By Pflueger's theorem \cite[Theorem 1.4]{Pfl17}, $W^r_d(\Gamma_g)=\cup_{i=1}^{\tilde{s}}(\kappa_i+T_{S_i})$ where $\kappa_i \in {\rm Jac}(\Gamma_g)$ and $T_{S_i}$ is a standard topological subtorus.  Since $T_{S_i}=t_{S_i}+\mathfrak{T}_{S_i}$ where $t_{S_i}=[-\sum_{j \notin S_i}(\mathfrak{o}_j)+(g-|S_i|) \cdot (w_g)]$, we have $W^r_d(\Gamma_g)=\cup_{i=1}^{\tilde{s}}(\kappa_i+t_{S_i}+\mathfrak{T}_{S_i})$. 
 Suppose that there are $s$ elements such that $(\kappa_i+t_{S_i}+\mathfrak{T}_{S_i}) \cap H \neq \emptyset$ (we know that $s \geq 1$ since $W^r_d(\Gamma_g) \cap H  \neq \emptyset$). By a suitable reordering, we assume that they correspond to indices one to $s$. For each $i$ from one to $s$, we set $A_i$ to be $\mathfrak{T}_{S_i}$ (along with the structure of a tropical Abelian subvariety.) 

 For each $i$ from one to $s$, since $(\kappa_i+t_{S_i}+\mathfrak{T}_{S_i}) \cap H \neq \emptyset$ there exists an element in $H$ that is contained in $(\kappa_i+t_{S_i}+\mathfrak{T}_{S_i})$. We set $\gamma_i$ to be any such element and note that $\kappa_i+t_{S_i}-\gamma_i \in \mathfrak{T}_{S_i}$.  Hence, $\kappa_i+t_{S_i}+\mathfrak{T}_{S_i}=\gamma_i+\mathfrak{T}_{S_i} \subseteq W^r_d(\Gamma_g)$  and $W^r_d(\Gamma_g) \cap H=(\cup_{i=1}^{s} (\gamma_i+A_i)) \cap H=\cup_{i=1}^{s} ((\gamma_i+A_i) \cap H)$.  An elementary fact from group theory \cite[Problem IV.6.a-13 ]{Mir95}  tells us that $(\gamma_i+A_i) \cap H=\gamma_i+(A_i \cap H)$ (note that $\gamma_i \in H$) and hence, $W^r_d(\Gamma_g) \cap H=\cup_{i=1}^{s} (\gamma_i+(A_i \cap H))$. We conclude that the tropical Abelian subvarieties $A_1,\dots,A_s$ and the corresponding translates $\gamma_1,\dots,\gamma_s \in H$ satisfy both conditions in Theorem \ref{langbncol_theo}. \end{proof}

 \begin{example}\label{langgen_ex} \rm
Consider a Brill-Noether general chain of three loops, denoted by $\Gamma_3$. Furthermore, suppose that the edge lengths $\ell_2$ (the length of the loop $L_2$) and $\ell(v_2w_2)$ are non-commensurable. We refer to \cite[Definition 4.1]{CooDraPayRob12} for a description of Brill-Noether general chains of loops. Since $\Gamma_3$ is Brill-Noether general, $W^r_d$ is non-empty if and only if $\rho(g,r,d)=g-(r+1)(g-d+r) \geq 0$. Hence, the only non-empty $W^r_d$s for integers $r$ and $d$ such that $0 \leq d \leq 4$ (note that $2g-2=4$) and $0 \leq r \leq d$ are $W^0_0,W^{0}_1,W^0_2,W^0_3,W^0_4,W^1_3, W^1_4$ and $W^2_4$. Note that $\Gamma_3$ is trigonal with the gonality determined by the $W^1_3$.  This $W^1_3$ is a union of three standard topological subtori, namely $A_1=[(w_1)+(v_2)-2(w_3)]+T_{\{3\}},~A_2=[(w_1)+(v_3)-2(w_3)]+T_{\{2\}},~A_3=[(w_2)+(v_3)-2(w_3)]+T_{\{1\}}$.  This decomposition can be computed either by inspection or by using the algorithm in \cite[Section 3]{Pfl17}. 

Consider the subgroup $H$ of ${\rm Jac}(\Gamma_3)$ generated by $[D_1]$ and $[D_2]$ where $D_1=(w_1)+(v_2)-2(w_3)$ and $D_2=(w_1)-(w_3)$. The subgroup $H$ has non-empty intersection with $A_1$ since $[D_1] \in A_1$. By examining the Pflueger reduced divisor of $m_1D_1+m_2D_2$ for integers $m_1,~m_2$ and using the assumption that $\ell_2$ and $\ell(v_2w_2)$ are non-commensurable, we find that $H$ intersects $A_2$ non-trivially but $H \cap A_3=\emptyset$. For $\gamma_1=[D_1]$ and $\gamma_2=[D_1+D_2]$, we have the decomposition $W^1_3 \cap H=(\gamma_1+(A_1 \cap H)) \cap (\gamma_2+(A_2 \cap H))$ as asserted by Theorem \ref{langbncol_theo}.
\qed 
\end{example}

 \section{Rationality of Poincar\'e Series of Divisors on Chains of Loops}\label{poinratcol_sect}



\subsection{Proof of Theorem \ref{poinratcol_theo}}

The proof is analogous to the case of algebraic curves \cite[Theorem 4.1]{Cut03} with Theorem \ref{langbncol_theo} playing the role of Lang's Conjecture.   Recall that by definition 

\begin{center}

$P_{\Gamma_g}(z_1,\dots,z_k)=\sum_{(n_1,\dots,n_k) \in \mathbb{N}^k}(r_{\Gamma_g}(n_1D_1+\dots+n_kD_k)+1)z_1^{n_1} \cdots z_k^{n_k}$.

\end{center}

Let $d_i={\rm deg}(D_i)$ for $i$ from one to $k$. We decompose  $P_{\Gamma_g}$ into graded pieces based on the weighted degree (with weights $d_1,\dots,d_k$ on $z_1\dots,z_k$ respectively) as follows.
Let $Q^{(l)}_{\Gamma_g}=\{ (n_1,\dots,n_k)  \in \mathbb{N}^k|~\sum_{i=1}^{k} n_id_i=l\}$, we define

\begin{center} $P^{(l)}_{\Gamma_g}(z_1,\dots,z_k)=\sum_{(n_1,\dots,n_k )\in Q^{(l)}_{\Gamma_g}}(r_{\Gamma_g}(n_1D_1+\dots+n_kD_k)+1)z_1^{n_1} \cdots z_k^{n_k}$.
\end{center} 




By construction, $P_{\Gamma_g}=\sum_{l \in \mathbb{Z}} P^{(l)}_{\Gamma_g}$.  As in the case of graphs, note that $P^{(l)}_{\Gamma_g}=0$ for all $l<0$ and 

\begin{center}
$\sum_{l=2g-1}^{\infty} P^{(l)}_{\Gamma_g}(z_1,\dots,z_k)=\sum_{(n_1,\dots,n_k) \in \mathbb{N}^k,~\sum_{i=1}^{k}n_i \cdot d_i \geq 2g-1} (r_{\Gamma_g}(n_1D_1+\dots+n_kD_k)+1)z_1^{n_1} \cdots z_k^{n_k}=\sum_{(n_1,\dots,n_k) \in \mathbb{N}^k,~\sum_{i=1}^{k}n_i \cdot d_i \geq 2g-1}(\sum_{i=1}^{k}n_i \cdot d_i-g+1) z_1^{n_1} \cdots z_k^{n_k}$.  
\end{center}

The second equality invokes the Riemann-Roch theorem for tropical curves \cite[Corollary 3.8]{GatKer08}, \cite[Theorem 7.4]{MikZha08} and the rationality of this series is exactly as in the case of graphs (see Section \ref{poinlarge_subsect} for more details).

Next, we consider $P^{(l)}_{\Gamma_g}$ for $l$ between zero and $2g-2$. Let $\bar{D_i}=D_i-d_i \cdot (w_g)$. We define the group homomorphism $\bar{\phi}_{\Gamma_g}: \mathbb{Z}^k \rightarrow {\rm Jac}(\Gamma_g)$ as $\bar{\phi}_{\Gamma_g}(m_1,\dots,m_k)=[\sum_{i=1}^{k}m_i \bar{D_i}]$.   Recall that $W^r_d(\Gamma_g)=\{ [D] \in {\rm Jac}(\Gamma_g)|~r_{\Gamma_g}(D+d \cdot (w_g)) \geq r\}$. For integers $r,~d$, we refine $Q^{(l)}_{\Gamma_g}$ to the subset:

\begin{center}
$Q^{(r,d)}_{\Gamma_g}= \{(n_1,\dots,n_k) \in \mathbb{Z}^k|~ \sum_{i=1}^{k}n_id_i=d,\bar{\phi}_{\Gamma_g}(n_1,\dots,n_k) \in W^r_d(\Gamma_g)\}$.
\end{center}

Note that $(n_1,\dots,n_k) \in Q^{(r,d)}_{\Gamma_g}$ if and only if  ${\rm deg}(\sum_{i=1}^{k}n_i D_i)=\sum_{i=1}^{k}n_i \cdot d_i=d$ and $r_{\Gamma_g}(\sum_{i=1}^{k}n_i D_i) \geq r$. We further refine $P^{(l)}_{\Gamma_g}$ as follows: 

\begin{center}
$P^{(r,d)}_{\Gamma_g}(z_1,\dots,z_k)=\sum_{(n_1,\dots,n_k) \in \mathbb{N}^k \cap Q^{(r,d)}_{\Gamma_g}}z_1^{n_1} \cdots z_k^{n_k}$.
\end{center}

For every non-negative integer $l$, we have $P^{(l)}_{\Gamma_g}=\sum_{r=0}^{l} P^{(r,l)}_{\Gamma_g}$. To see this, first observe that the support of both series is contained in $\{(n_1,\dots,n_k) \in \mathbb{N}^k|~\sum_{i=1}^{k}n_i \cdot d_i=l\}$. If the rank of $\sum_{i=1}^{k}n_iD_i$ is minus one then the coefficient of $z_1^{n_1}\cdots z_k^{n_k}$ is zero on both sides.  On the other hand, if  the rank of $\sum_{i=1}^{k}n_iD_i$  is non-negative, then the coefficient of $z_1^{n_1} \cdots z_k^{n_k}$ is $r_{\Gamma_g}(\sum_{i=1}^{k}n_iD_i)+1$ on both sides. This holds by definition for the LHS and for the RHS, note that $z_1^{n_1} \cdots z_k^{n_k}$ appears in  $P^{(r,l)}_{\Gamma_g}$ (with coefficient one) precisely for $r$ from $0,\dots,r_{\Gamma_g}(\sum_{i=1}^{k}n_iD_i)$ (note that $r_{\Gamma_g}(\sum_{i=1}^{k}n_iD_i) \leq l$).   Hence, it suffices to show that $P^{(r,d)}_{\Gamma_g}$ is rational for every choice of $(r,d)$.

We show the rationality of $P^{(r,d)}_{\Gamma_g}$ via Theorem \ref{langbncol_theo}. More precisely, via Theorem \ref{langbncol_theo}, we show that each non-empty $Q^{(r,d)}_{\Gamma_g}$ is a finite union of affine sublattices of $\mathbb{Z}^k$ where by an affine sublattice, we mean a coset of a sublattice of $\mathbb{Z}^k$. This implies that $P^{(r,d)}_{\Gamma_g}$ is an integer combination of finitely many lattice point enumerating functions of rational polyhedra and its rationality follows from their rationality \cite[Theorem 3.1]{Bar99}, \cite[Corollary 7.6]{CutHerReg03}.

\begin{proposition}\label{qrddec_prop} Let  $(r,d)$  be a pair of integers. If $Q^{(r,d)}_{\Gamma_g} \neq \emptyset$, then it is a finite union of affine sublattices of $\mathbb{Z}^k$. \end{proposition}

\begin{proof}
Let $H$ be the subgroup  of ${\rm Jac}(\Gamma_g)$ generated by $[\bar{D_1}],\dots,[\bar{D_k}]$.  Since the image of $\bar{\phi}_{\Gamma_g}$ is equal to $H$, we have $Q^{(r,d)}_{\Gamma_g}=\{(n_1,\dots,n_k) \in \mathbb{Z}^k|~\sum_{i=1}^{k}n_id_i=d,~\bar{\phi}_{\Gamma_g}(n_1,\dots,n_k) \in W^r_d(\Gamma_g) \cap H\}$. If $Q^{(r,d)}_{\Gamma_g} \neq \emptyset$ then $W^r_d(\Gamma_g) \cap H \neq \emptyset$ and  by Theorem \ref{langbncol_theo}, we have   $Q^{(r,d)}_{\Gamma_g}=\{(n_1,\dots,n_k) \in \mathbb{Z}^k|~~\sum_{i=1}^{k}n_id_i=d, \bar{\phi}_{\Gamma_g}(n_1,\dots,n_k) \in \cup_{i=1}^{s}(\gamma_i+(A_i \cap H))\}$ where each ~$\gamma_i \in H$ and each $A_i$ is a tropical Abelian subvariety of ${\rm Jac}(\Gamma_g)$.

We analyse the fibre of $\bar{\phi}_{\Gamma_g}$ over each $\gamma_i+(A_i \cap H)$.  Let $\Lambda_i$ be the fibre of $\bar{\phi}_{\Gamma_g}$ over $A_i \cap H$. Since $\bar{\phi}_{\Gamma_g}$ is a group homomorphism and $A_i$ is a subgroup of ${\rm Jac}(\Gamma_g)$, we know from elementary group theory that $\Lambda_i$ is a subgroup of $\mathbb{Z}^k$. The only subgroups of $\mathbb{Z}^k$ are sublattices and hence, $\Lambda_i$ is a sublattice of $\mathbb{Z}^k$.  Since $\gamma_i \in H$, we know that there is a point ${\bf m}_i \in \mathbb{Z}^k$ whose image under $\bar{\phi}_{\Gamma_g}$ is $\gamma_i$.  Hence, the fibre of $\bar{\phi}_{\Gamma_g}$ over $\gamma_i+(A_i \cap H)$ is ${\bf m}_i+\Lambda_i$. Furthermore, $Q^{(r,d)}_{\Gamma_g}$ is the intersection of the union of affine lattices $\cup_{i=1}^{s}({\bf m_i}+\Lambda_i)$ with the affine lattice $\{(n_1,\dots,n_k) \in \mathbb{Z}^k|~\sum_{i=1}^kn_id_i=d\}$. Since this intersection is non-empty, we know from elementary group theory that this is a finite union of affine sublattices. 
\end{proof}

By Proposition \ref{qrddec_prop}, consider the decomposition $Q^{(r,d)}_{\Gamma_g}=\cup_{i \in \mathcal{F}} \hat{\Lambda}_i$ where each $\hat{\Lambda}_i={\bf q_i}+\Lambda_i$ is an affine sublattice, i.e. ${\bf q_i} \in \mathbb{Z}^k$, $\Lambda_i$ is a sublattice of $\mathbb{Z}^k$ and $\mathcal{F}$ is a finite set.  Note that a non-empty finite intersection of affine sublattices is  also an affine sublattice. For a finite subset $S$ of $\mathcal{F}$, let $\hat{\Lambda}_S$ denote $\cap_{ i \in S} \hat{\Lambda}_i$. We omit the brackets in the subscript while denoting singletons.  By the inclusion-exclusion formula, we have:

\begin{center}

$P^{(r,d)}_{\Gamma_g}(z_1,\dots,z_k)= $\\
$\sum_{i \in \mathcal{F}} \sum_{(n_1,\dots,n_k) \in \mathbb{N}^k \cap \hat{\Lambda}_i} z_1^{n_1} \cdots z_k^{n_k}-\sum_{|S|=2}  \sum_{(n_1,\dots,n_k) \in \mathbb{N}^k \cap \hat{\Lambda}_S} z_1^{n_1} \cdots z_k^{n_k}+\dots+(-1)^{|\mathcal{F}|+1}  \sum_{(n_1,\dots,n_k) \in \mathbb{N}^k \cap \hat{\Lambda}_\mathcal{F}} z_1^{n_1} \cdots z_k^{n_k}.$

\end{center}

Note that if $\mathbb{N}^k \cap \hat{\Lambda}_S$ is empty  for some subset $S$, then the corresponding sum is taken to be zero. Each term in the above decomposition is (an affine) lattice point enumerating function of a rational polyhedron and is hence rational (\cite[Theorem 3.1]{Bar99},  \cite[Corollary 7.6]{CutHerReg03}).  Hence,  $P^{(r,d)}_{\Gamma_g}$ is also rational. Furthermore, since
\begin{center} $P_{\Gamma_g}=\sum_{d=0}^{2g-2}\sum_{r=0}^{d}P^{(r,d)}_{\Gamma_g}+\sum_{d=2g-1}^{\infty}P^{(d)}_{\Gamma_g}$. \end{center}
We conclude that $P_{\Gamma_g}$ is itself rational. The claim that this rational function agrees with the corresponding power series at every point where the power series is absolutely convergent follows from the corresponding property for each lattice point enumerating function in the sum. 
\qed

\subsection{Explicit Construction of the Affine Sublattices} \label{afflatcon_subsect}

In the following, we construct the affine sublattices appearing in the decomposition of $Q^{(r,d)}_{\Gamma_g}$ from Proposition \ref{qrddec_prop}. The key to this is the following characterisation of cosets of standard subgroup tori. 

\begin{proposition}\label{stasubchar_prop}
Let $t \in {\rm Jac}(\Gamma_g)$ and let $(\sum_{j=1}^{g}\beta_j(t))-g \cdot (w_g)$ be its Pflueger reduced divisor. A divisor class is contained in $t+\mathfrak{T}_S$ if and only if for every $j \notin S$, the $j$-th component its Pflueger reduced divisor is equal to $\beta_j(t)$.
\end{proposition}
\begin{proof}
$(\Rightarrow)$ By the reduction algorithm, we verify that the $j$-th component of the Pflueger reduced divisor of any element in $t+\mathfrak{T}_S$ for $j \notin S$ is $\beta_j(t)$.

$(\Leftarrow)$  Suppose that a divisor class $[D]$ has Pflueger reduced divisor of the form $\sum_{j \in S}(\xi_j)+\sum_{j \notin S}(\beta_j(t))-g \cdot (w_g)$. Consider the divisor class $[D'] \in \mathfrak{T}_S$ whose Pflueger reduced divisor is equal to 
$\sum_{j \in S}(\xi'_j)+\sum_{j \notin S}(\mathfrak{o}_j)-g \cdot (w_g)$ where each $\xi'_j$ is the unique point in $L_j$ whose associated divisor is linearly equivalent to $(\xi_j)-(\beta_j(t))+j \cdot (w_j)-(j-1) \cdot (v_j)$. We verify, using the reduction algorithm, that the Pflueger reduced divisor of $t+[D']$ is  $\sum_{j \in S}(\xi_j)+\sum_{j \notin S}(\beta_j(t))-g \cdot (w_g)$. Hence, $[D]=t+[D'] \in t+\mathfrak{T}_S$. 
\end{proof}

Next, we compute the Pflueger reduced divisor of elements in the group generated by $[D_1],\dots,[D_k]$. The following proposition will be turn out to be useful. Let $\Gamma_1$ be a single loop (of length $\ell_1$)  with a marked point $w_1$. We identify the points of $\Gamma_1$ with elements in $\mathbb{R}/(\ell_1 \cdot \mathbb{Z})$ by taking a point $q \in \Gamma_1$ to the element in $\mathbb{R}/(\ell_1 \cdot \mathbb{Z})$ corresponding to its  anticlockwise distance from $w_1$. Note that $\mathbb{R}/(\ell_1 \cdot \mathbb{Z})$ is naturally a $\mathbb{Z}$-module (via the multiplication action). Let $D=\sum_{i=1}^{N-1} \alpha_i (q_i)$ be a non-zero principal divisor on $\Gamma_1$ where each $q_i$ is distinct and is in increasing order with respect to its anticlockwise distance from $w_1$.    Let $\langle q_i \rangle$ be the element of $\mathbb{R}/(\ell_1 \cdot \mathbb{Z})$ corresponding to $q_i$  \footnote{Note that Pflueger \cite{Pfl17} uses this notation in a slightly different sense.}.





\begin{proposition}\label{prinlindep_prop}
The support of $D$, as a subset of $\mathbb{R}/(\ell_1 \cdot \mathbb{Z})$, is linearly dependent over $\mathbb{Z}$. Furthermore, $\sum_{i=1}^{N-1} \alpha_i \langle q_i \rangle \equiv 0 \pmod {\ell_1}$. 
\end{proposition}

\begin{proof}
Since $D$ is a principal divisor on $\Gamma_1$, there exists a rational function $f_{\Gamma_1}$ on $\Gamma_1$ whose divisor is $D$. Since every point of $\Gamma_1$ has valence two,  the bend locus of $f_{\Gamma_1}$ is precisely the support of $D$. Let $(\zeta_1,\dots,\zeta_{N-1})$  be the sequence of anticlockwise distances of points in the support of $D$ to $w_1$ in increasing order. Set $\zeta_0=0$. Let $s_i \in \mathbb{Z}$ be the  slope of $f_{\Gamma_1}$ along the segment $(\zeta_{i \mod N}, \zeta_{(i+1)\mod N})$, if $\zeta_{i \mod N}$ and $\zeta_{(i+1)\mod N}$ are distinct and zero, otherwise (this can happen only if $i=0$).   Since $f_{\Gamma_1}$ is piecewise linear,  integrating the differential of $f_{\Gamma_1}$ around $\Gamma_1$ and using the fundamental theorem of calculus, we have:
\begin{center}
$\sum_{i=0}^{N-2} s_i(\zeta_{i+1}-\zeta_i)+s_{N-1}(\ell_1-\zeta_{N-1})=0$
\end{center}

 Rearranging terms, we obtain:

\begin{center}
$\sum_{i=0}^{N-2} (s_i-s_{i+1}) \zeta_{i+1}=-s_{N-1} \ell_1$
\end{center}

Hence,  the projection of $\zeta_1,\dots \zeta_{N-1}$ onto $\mathbb{R}/(\ell_1 \cdot \mathbb{Z})$ (this is the set $\{\langle q_i \rangle \}_{i=1}^{N-1}$ ) is linearly dependent over $\mathbb{Z}$.  Furthermore, suppose that $q_1 \neq w_1$ then we note that since each $q_i$ is distinct and that $\alpha_i=s_{i}-s_{i-1}$ for each $i$ from one to $N-1$ to conclude that $\sum_{i=1}^{N-1} \alpha_i \langle q_i \rangle \equiv 0 (\mod \ell_1)$. On the other hand, if $q_1=w_1$ then $\alpha_i=s_{i}-s_{i-1}$ for each $i$ from two to $N-1$ and  since $\zeta_1=0$, we obtain $\sum_{i=1}^{N-2} (s_i-s_{i+1}) \zeta_{i+1}=-s_{N-1} \ell_1$ and hence, $\sum_{i=1}^{N-1} \alpha_i \langle q_i \rangle \equiv 0 (\mod \ell_1)$. 

\end{proof}

 Let $[D_1],\dots,[D_k] \in {\rm Jac}(\Gamma_g)$.  Suppose that $D_i=\sum_{j=1}^{g} (\xi_{i,j})-g \cdot (w_g)$ be the Pflueger reduced divisor of $[D_i]$ for each $i$ from one to $k$. Let $\tau_{i,j}$ be the element in $\mathbb{R}/(\ell_j \cdot \mathbb{Z})$ corresponding to $\xi_{i,j}$.

\begin{lemma}\label{pflgred_lem}
For each $j$ from one to $g$, the $j$-th component $\psi_j$ of the Pflueger reduced divisor of $\sum_{i=1}^{k} \alpha_i[D_i]$ satisfies the following equation:

\begin{center}

$\langle \psi_j  \rangle \equiv \sum_{i=1}^{k} \alpha_i \tau_{i,j}+(j-1) (\sum_{i=1}^{k} \alpha_i-1) \ell(v_jw_j) (\mod \ell_j)$

\end{center}

where $\langle \psi_j  \rangle$ is the element in $\mathbb{R}/(\ell_j \cdot \mathbb{Z})$ corresponding to $\psi_j$.
\end{lemma}

\begin{proof}
We apply the reduction algorithm to the divisor $\sum_{i=1}^{k} \alpha_i D_i$ to deduce that the $j$-th component $\psi_j$ of its Pflueger reduced divisor is  the unique point  in $L_j$ that is linearly equivalent to  $\sum_{i=1}^{k} \alpha_i (\xi_{i,j})-(j-1)(1-\sum_{i=1}^{k} \alpha_i)(v_j)+j(1-\sum_{i=1}^{k} \alpha_i)(w_j)$.  Hence, $(\psi_j)-\sum_{i=1}^{k} \alpha_i (\xi_{i,j})+(j-1)(1-\sum_{i=1}^{k} \alpha_i)(v_j)-j(1-\sum_{i=1}^{k} \alpha_i)(w_j)$ is a principal divisor on $L_j$.

We apply Proposition \ref{prinlindep_prop} to this principal divisor to obtain the congruence $\langle \psi_j  \rangle \equiv \sum_{i=1}^{k} \alpha_i \tau_{i,j}+(j-1) (\sum_{i=1}^{k} \alpha_i-1) \ell(v_jw_j) (\mod \ell_j)$.  Note that the collection of points $v_j,~w_j$ and $\{\tau_{i,j}\}_{i=1}^{k}$ need not be distinct but this does not affect the congruence. 
\end{proof}

As a corollary, we obtain the following explicit characterisation of the affine lattices in the decomposition of $Q^{(r,d)}_{\Gamma_g}$ according to Proposition \ref{qrddec_prop}.

\begin{corollary}\label{affinetor_cor}

Suppose that the affine lattice $\hat{\Lambda}$ is the fibre over the coset $t+\mathfrak{T}_S$ of $\bar{\phi}_{\Gamma_g}$, then $\hat{\Lambda}$ as a set consists of points $(\alpha_1,\dots,\alpha_k) \in \mathbb{Z}^k$ satisfying the linear system of equations:
\begin{equation}\label{affsys_eq}
\sum_{i=1}^{k} \alpha_i \tau_{i,j}+(j-1)(\sum_{i=1}^{k} \alpha_i) \ell(v_jw_j) \equiv \beta_j(t)+(j-1)\ell(v_jw_j) {\mod \ell_j}
\end{equation}

for every $j \notin S$, $\beta_j(t)$ is as in Proposition \ref{stasubchar_prop} and $\tau_{i,j}$ is as in Lemma \ref{pflgred_lem}. 

\end{corollary}

We can construct $\hat{\Lambda}$  as follows:  consider the lattice $\Lambda$ defined by the linear system of equations:

\begin{equation}\label{linsys_eq}
\sum_{i=1}^{k} \alpha_i \tau_{i,j}+(j-1)(\sum_{i=1}^{k} \alpha_i) \ell(v_jw_j) \equiv 0{\mod \ell_j}
\end{equation}

Suppose that ${\bf q} \in \mathbb{Z}^k$ is a solution to Equation (\ref{affsys_eq}).  We have $\hat{\Lambda}={\bf q}+\Lambda$.

In the following, we describe the lattice $\Lambda$. Let $\mathcal{M}_j$ be the submodule of $\mathbb{R}/(\ell_j \cdot \mathbb{Z})$ (as a $\mathbb{Z}$-module) generated by $\tau_{1,j},\dots,\tau_{k,j},\ell(v_jw_j)$.  Let the lattice $\Lambda_j$ be the intersection of the syzygy module (a sublattice of $\mathbb{Z}^{k+1}$) of the finitely generated module $\mathcal{M}_j$ (with respect to the induced generating set) and the lattice $\{(y_1,\dots,y_{k+1})|~y_{k+1}=(j-1)\sum_{i=1}^{k}y_i\} \cap \mathbb{Z}^{k+1}$.  
The lattice $\Lambda=\cap_{j \notin S}\Lambda_j$. Since each $\mathcal{M}_j$  is a finitely generated Abelian group, the syzygy module can be described in terms of the syzygies with respect to a standard generating set (a basis for the free summand and a generator for each cyclic summand) and a homomorphism between two free Abelian groups. We omit the details.

 \section{An Algorithm to Compute the Poincar\'e Series of Divisors on a Chain of Loops}\label{poicolalg_sect}

We present an algorithm to compute the Poincar\'e series of divisors on a chain of loops (given in terms of the edge lengths and the divisors). The algorithm is presented at a ``high level" with the purpose of aiding practical computation. The algorithm relies on three key computations: i. Decomposing a $W^r_d$ into a finite union of standard topological subtori, given $r$ and $d$, from \cite[Section 3]{Pfl17}.  ii. Computing the affine lattices in the decomposition of $Q^{(r,d)}_{\Gamma_g}$, as described in Subsection \ref{afflatcon_subsect}. iii. Computing the lattice point enumerating function of a rational polyhedron (given as an $\mathcal{H}$-polyhedron) with respect to an affine sublattice of $\mathbb{Z}^k$ (given in terms of a basis for the underlying lattice along with a translate) from \cite[Theorem 4.4]{Bar99} (also, see the second remark after Theorem 5.3 of this paper). 


\begin{breakablealgorithm} \label{Poincol_algo}
\caption{An Algorithm to Compute the Poincar\'e Series of Divisors on a Chain of Loops.}
\begin{algorithmic}

\State {\bf Input:} A chain of loops $\Gamma_g$ of genus $g$ and divisors $D_1,\dots,D_k$. 

\ForEach{$i$ from one to $k$} 
 \State Compute $d_i$ the degree of $D_i$ and $\bar{D}_i=D_i-d_i(w_g)$.
 \EndFor
 
 \State Set $P:=0$

\ForEach{Pair of integers $r,~d$ such that $d \in [0,\dots,2g-2]$ and $r \in [0,\dots,d]$}
 
  \If{$W^r_d \neq \emptyset$}
      \State Decompose $W^r_d$ into a finite union $\cup_i (\kappa_i+T_{S_i})$ of standard topological tori. 
      \State Set $\mathcal{F}:=\emptyset$.
      \ForEach{Subtorus $\kappa_i+T_{S_i}$ in the decomposition}
         \If{$(\kappa_i+T_{S_i}) \cap \langle [\bar{D}_1],\dots,[\bar{D}_k] \rangle \neq \emptyset$}
              \State Include $i$ into $\mathcal{F}$.
              \State Compute the affine lattice $\bar{\Lambda}_i=\bar{\phi}_{\Gamma_g}^{-1}(\kappa_i+T_{S_i})$. 
              \State Compute $\hat{\Lambda}_i:=\bar{\Lambda}_i \cap \{(n_1,\dots,n_k) \in \mathbb{Z}^k|~\sum_{i=1}^{k}n_id_i=d\}$.   
         \EndIf
       \EndFor
     \ForEach{Non-empty subset $U$ of $\mathcal{F}$}
       \State Compute $\hat{\Lambda}_{U}=\cap_{i \in U}\hat{\Lambda}_i$.
       \If{$\hat{\Lambda}_{U} \neq \emptyset$}
        \State Compute the lattice point enumerating function $f_U$ of $\mathbb{R}^k_{\geq 0}$ with respect to the affine lattice $\hat{\Lambda}_{U}$. 
       \Else
         \State Set $f_U:=0$. 
       \EndIf
     \EndFor
     \State Set $P^{(r,d)}:=\sum_{U \subseteq \mathcal{F}, ~|U|=1} f_U-\sum_{U \subseteq \mathcal{F}, ~|U|=2}f_U+\sum_{U \subseteq \mathcal{F}, ~|U|=3}f_U+\cdots+{-1}^{(|\mathcal{F}|+1)}f_{\mathcal{F}}$.
  \EndIf
  
   \If{$W^r_d = \emptyset$} 
    \State Set  $P^{(r,d)}:=0$.
   \EndIf
   \State Add $P^{(r,d)}$ to $P$.
\EndFor
\State  Compute the lattice point enumerating function  $f_Q$ (with respect to $\mathbb{Z}^k$) of $Q=\{(n_1,\dots,n_k) \in \mathbb{R}^k|~n_i \geq 0 \text{ for all }  i, \sum_{i=1}^{k}n_id_i \geq 2g-1\}$.

\State Add $(\sum_{i=1}^{k} d_i\partial_{z_i} -(g-1))f_Q$ to $P$.

  \State {\bf Output:} $P$.

\end{algorithmic}
\end{breakablealgorithm}

\subsection{An Example}
We illustrate the algorithm in a simple example. 
\begin{example} \rm
Consider a generic chain of three loops as in Example \ref{langgen_ex}.  Additionally, suppose that the ratio of the lengths $\ell(v_3w_3)$ and $\ell_3$ is $1/5$ (by \cite[Definition 4.1]{CooDraPayRob12} this is permissible) . 
Let $D_1=(w_1)+(v_2)+(w_3)$ and $D_2=(w_1)-(w_3)$.  Hence, $d_1=3$ and $d_2=0$. In the following, we compute the Poincar\'e series $P_{\Gamma_3}$.
 We start by computing the set $Q^{(r,d)}_{\Gamma_3}$ (we shall use the shorthand $Q^{(r,d)}$) for each pair $(r,d)$ such that $W^r_d$ is non-empty and $0 \leq d \leq 4$. By degree considerations, we find that among these only $Q^{(0,0)},~Q^{(0,3)}$ and $Q^{(1,3)}$ are non-empty.  Using a method as in Corollary \ref{affinetor_cor}, we determine $Q^{(0,0)}=\{(0,0)\},~Q^{(0,3)}=\{(1,\eta)|~\eta \in \mathbb{Z}\}$ and $Q^{(1,3)}=\{(1,0)\} \cup \{(1,5 \eta+1)|~\eta \in \mathbb{Z}\}$.
For instance, we determine $Q^{(1,3)}$ as follows. Since $m_1d_1+m_2d_2=3$, we know that $m_1=1$. The Pflueger reduced divisor of $\bar{D}_1+m_2 \bar{D}_2$ is $(w_1)+(\xi_2)+(\xi_3)$ where $(\xi_2) \sim (m_2+1)(v_2)-m_2(w_2)$ and $(\xi_3) \sim m_2(v_3)-(m_2-1)(w_3)$. We compare this with the Pflueger reduced divisors of elements in $A_2$ and $A_3$. This combined with the assumption about the relation between $\ell_i$ and $\ell(v_iw_i)$ for $i=2,3$ allows us to compute $Q^{(1,3)}$.

With this information at hand, we compute the Poincar\'e series. In particular, $P^{(0,0)}=1, P^{(0,3)}=\dfrac{z_1}{(1-z_2)}$ and $P^{(1,3)}=z_1+\dfrac{z_1z_2}{(1-z_2^5)}$.  The series $\sum_{l>2g-2} P^{(l)}$ can be computed using the procedure described in Subsection \ref{poinlarge_subsect}. The associated lattice point enumerating function $f_Q=\dfrac{z_1^2}{(1-z_1)(1-z_2)}$ and $\sum_{l>2g-2} P^{(l)}=(3 \partial_{z_1}-2)(f_Q)$. Hence, the  Poincar\'e series $P_{\Gamma_3}=1+z_1+\dfrac{z_1}{(1-z_2)}+\dfrac{z_1z_2}{(1-z_2^5)}+(3 \partial_{z_1}-2)(f_Q)$.
\qed
\end{example}

\end{document}